\documentclass[10pt]{article}
\usepackage{amsfonts}
\usepackage{amsmath,amsthm,amssymb,latexsym,amscd,xcolor}
\usepackage{epsf}
\usepackage[pdftex]{graphicx}
\usepackage{tikz}
\usepackage{hyperref} 
\newtheorem{theorem}{Theorem}[section]
\newtheorem{lemma}[theorem]{Lemma}
\newtheorem{corollary}[theorem]{Corollary}

\theoremstyle{definition}

\theoremstyle{remark}
\newtheorem{remark}[theorem]{Remark}

\numberwithin{equation}{section}

\title{\bf  Bloch waves in crystals and periodic high contrast media}

\author{Robert Lipton\thanks{Department of Mathematics, Louisiana State University,
Baton Rouge, LA 70803, USA,
{\tt lipton@math.lsu.edu}}
\and
Robert Viator Jr.
\thanks{Department of Mathematics,
Louisiana State University,
Baton Rouge, LA 70803, USA,
{\tt rviato2@lsu.edu}}}

\date{}
\begin{document}
\maketitle
\begin{abstract}

Analytic representation formulas and power series are developed describing the band structure inside  periodic photonic and acoustic crystals made from high contrast inclusions. Central to this approach is the identification and utilization of a resonance spectrum for quasi-periodic source free modes. These modes are used to represent solution operators associated with electromagnetic and acoustic waves inside periodic high contrast media. Convergent power series for the Bloch wave spectrum is recovered from the representation formulas. Explicit conditions on the contrast are found that provide lower bounds on the convergence radius. These conditions are sufficient for the separation of spectral branches of the dispersion relation. 
\end{abstract}

\maketitle

\section{Introduction}
\label{introduction}

Recent decades have seen intense interest in wave propagation through high contrast periodic media.These materials have been studied both theoretically and experimentally and have been shown to exhibit unique optical, acoustic, and elastic properties \cite{J87}, \cite{Y}.
Here we develop analytic representation formulas and power series for dispersion relations describing wave propagation inside periodic crystals made from high contrast inclusions.  These tools are applied to investigate the propagation band structure as a function of the crystal geometry. 

Consider a Bloch wave $h(x)$ with Bloch eigenvalue $\omega^2$  propagating through a two or three dimensional crystal lattice characterized by the periodic coefficient $a(x)=a(x+p)$, $ p\in\mathbb{Z}^d$, $d=2,3$, with unit cell $Y=(0,1]^d$. The Bloch wave satisfies the differential equation,
\begin{equation}
-\nabla\cdot(a(x)\nabla h(x)) =\omega^2 h(x) ,\hbox{  $x\in\mathbb{R}^d$, $d=2,3$}
\label{Eigen1}
\end{equation}
\linebreak
together with the $\alpha$ quasi-periodicity condition $h(x+p)=h(x)e^{i\alpha\cdot p}$. Here $\alpha$ lies in the first Brillouin zone of the reciprocal lattice given by $Y^\star=(-\pi,\pi]^d$.  Equation \eqref{Eigen1} describes acoustic wave propagation through crystals and transverse magnetic (TM) wave propagation through a two dimensional photonic crystal.
We examine Bloch wave propagation through high contrast crystals made from  periodic configurations of two materials. One material occupies disjoint inclusions and is completely contained within each period cell and surrounded by the second material.
The coefficient is taken to be $1$ inside the inclusions and $k>0$ outside. The domain occupied by the union of all the inclusions $D_1,D_2,\ldots,D_n$ inside $Y$ is denoted by $D$, see figure \ref{plane}. 
The coefficient is specified on the unit period cell by $a(x)=(k \chi_{Y \setminus D} (x)+ \chi_{D}(x))$ where $\chi_{D}$ and $\chi_{Y \setminus D}$ are  indicator functions for the sets $D$ and $Y\setminus D$ and are extended by periodicity to $\mathbb{R}^d$. In this paper we consider periodic crystals made from finite collections of separated inclusions each with $C^{1,\gamma}$ boundary.


\begin{figure} 
\centering
\begin{tikzpicture}[xscale=1.0,yscale=1.0]
\draw [thick] (-2,-2) rectangle (3,3);
\draw [fill=orange,thick] (-0.2,-0.6) circle [radius=1.25];
\draw [fill=orange,thick] (2.2,2.0) circle [radius=0.6];
\node [right] at (1.9,2.0) {$D_2$};
\draw [fill=orange,thick] (-1,2.2) circle [radius=0.6];
\node [right] at (-1.3,2.2) {$D_6$};
\draw[fill=orange,thick](-1.2,1.05) ellipse (20pt and 10pt);
\node [right] at (-1.5,1.05) {$D_5$};
\draw[fill=orange,thick](2.2,-0.1) ellipse (10pt and 20pt);
\node [right] at (1.85,-0.1) {$D_4$};
\node [below] at (-0.12,-0.3) {$D_1$};
\node [below] at (0.2,2.0) {$Y\setminus D$};
\draw [fill=orange, thick] plot[ smooth cycle] coordinates{(1,0) (2,1) (1,2) (1.1,1.1)};
\node [right] at (1.1,1.0) {$D_3$};
\end{tikzpicture} 
\caption{{\bf Period Cell.}}
 \label{plane}
\end{figure}
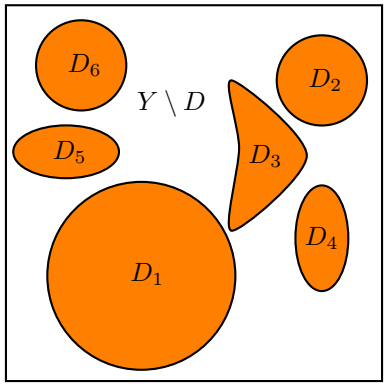

For each $\alpha\in Y^\star$ the Bloch eigenvalues $\omega^2$ are of finite multiplicity and denoted by $\lambda_j(k,\alpha)$, $j\in \mathbb{N}$. We develop power series expansions for each branch of the dispersion relation
\begin{equation}
\lambda_j(k,\alpha)=\omega^2,\hbox{ $j\in\mathbb{N}$}
\label{DispersionRelations}
\end{equation}
that are valid for $k$ in a neighborhood of infinity.

To proceed we complexify the problem and consider $k\in\mathbb{C}$.  Now $a(x)$ takes on complex values inside $Y\setminus D$ and the divergence form operator $-\nabla\cdot (k \chi_{Y \setminus D} + \chi_{D})\nabla$
is no longer uniformly elliptic. Our approach develops an explicit representation formula for $-\nabla\cdot (k \chi_{Y \setminus D} + \chi_{D})\nabla$ that holds for complex values of $k$.  We identify the subset $z={1}/{k}\in\Omega_0$ of $\mathbb{C}$ where this operator is invertible. The explicit formula shows that the solution  operator $(-\nabla\cdot (k \chi_{Y \setminus D} + \chi_{D})\nabla)^{-1}$ may be regarded more generally as a meromorphic operator valued function of $z$ for $z\in\Omega_0=\mathbb{C}\setminus S$, see section \ref{asymptotic}  and Lemma \ref{inverseoperator}. Here the set $S$ is discrete and consists of poles lying on the negative real axis with only one accumulation point at $z=-1$. For the problem treated here we expand about $z=0$ and the distance between $z=0$ and the set $S$ is used to bound the radius of convergence for the power series. The spectral representation for $-\nabla\cdot (k \chi_{Y \setminus D} + \chi_{D})\nabla$ follows from the existence of a complete orthonormal set of quasi-periodic functions associated with the {\em quasi-periodic resonances of the crystal}, i.e., quasi periodic functions $v$ and real eigenvalues $\lambda$ for which
\begin{eqnarray}
\label{sourcefree}
-\nabla\cdot (\chi_{D})\nabla v=-\lambda \Delta v.
\end{eqnarray}
These resonances are  shown to be connected to the spectra of Neumann-Poincar\'e operators  associated with quasi-periodic double layer potentials. For $\alpha=0$ these are the well known electrostatic resonances identified in \cite{BergmanES}, \cite{BergmanC}, \cite{MiltonES}, and \cite{Milton}. Both Neumann-Poincar\'e operators and associated electrostatic resonances have been the focus of  theoretical investigations \cite{Kang}, \cite{Shapero} and applied in analysis of plasmonic excitations for suspensions of noble metal particles \cite{Mayorgoz} and electrostatic breakdown \cite{JagerMosko}. The explicit spectral representation for the operator  $-\nabla\cdot (k \chi_{Y \setminus D} + \chi_{D})\nabla$ is crucial for elucidating the interaction between the contrast $k$ and the quasi-periodic resonances of the crystal, see \eqref{inverse},  \eqref{noninvertability}, and \eqref{representation}.  
The spectral representation is applied to analytically continue  the band structure $\lambda_j(k,\alpha)=\omega^2$, $j\in\mathbb{N}$, $\alpha\in Y^\star$ for $k$ onto $\mathbb{C}$, see Theorem \ref{extension}. 
Application of  the contour integral formula for spectral projections \cite{SzNagy}, \cite{TKato1}, \cite{TKato2} delivers an analytic representation formula for the band structure, see section \ref{asymptotic}.  We apply perturbation theory  in section \ref{asymptotic} together with a calculation provided in section \ref{derivation} to find an explicit formula for the radii of convergence for the power series  $\lambda_j(k,\alpha)$  about $1/k=0$. The formula shows that the radius of convergence and separation between different branches of the dispersion relation are determined by: 1) the distance of the origin to the nearest pole $z^\ast$ of $(-\nabla\cdot (k \chi_{Y \setminus D} + \chi_{D})\nabla)^{-1}$, and 2) the separation between distinct eigenvalues in the $z=1/k\rightarrow 0$ limit, see Theorems \ref{separationandraduus-alphanotzero} and \ref{separationandraduus-alphazero}. These theorems provide conditions on the contrast guaranteeing the separation of spectral bands that depend explicitly upon  $z^\ast$,  $j\in\mathbb{N}$ and $\alpha\in Y^\star$. Error estimates for series truncated after $N$ terms follow directly from the formulation. 

Next we apply these results and develop bounds on the convergence radii for a wide class of inclusions called buffered geometries. A buffered geometry  is described by any randomly distributed  collection of separated inclusions with a prescribed minimum distance of separation between inclusions, see section \ref{radiusgeneralshape}. For these geometries we demonstrate that the poles of $(-\nabla\cdot (k \chi_{Y \setminus D} + \chi_{D})\nabla)^{-1}$ associated with the quasi-periodic spectra are bounded away from the origin uniformly for $\alpha\in Y^*$. 
The quasi-periodic spectra $\{\mu_i\}_{i\in\mathbb{N}}$ associated with a buffered geometry is shown to lie inside the interval $-1/2<{\mu}^-\leq \mu_i\leq 1/2,$ for every $\alpha\in Y^\star$, see Theorem \ref{lowerboundrho} and Corollary \ref{theta}. 
The lower bound ${\mu}^-$ is independent of $\alpha\in Y^\star$ and depends explicitly on the geometry of the inclusions expressed in terms of the norm of the Dirichlet to Neumann map of each inclusion shape. This control insures that the associated poles of $(-\nabla\cdot (k \chi_{Y \setminus D} + \chi_{D})\nabla)^{-1}$ are uniformly bounded away from the origin and insures a nonzero radius of convergence for the power series representation for the band structure $\lambda_j(k,\alpha)=\omega^2$ for each $j\in\mathbb{N}$ and $\alpha\in Y^\star$, see Theorems \ref{separationandraduus-alphanotzero} and \ref{separationandraduus-alphazero}. In section \ref{radiusmultiplescatterers} we apply these observations to periodic assemblages of buffered disks. Here a buffered disk assemblage is characterized by a period filled with a randomly distributed collection of $N$ disks of equal radius separated by a minimum distance. For this case we recover explicit formulas for the radii of convergence of the power series expansion for $\lambda_j(k,\alpha)$ and explicit conditions for the separation of spectral branches in terms of the minimum distance between disks in the assemblage. It is important to emphasize that the results on separation of spectra  and convergence of power series are not asymptotic results but are valid for an explicitly delineated regime of finite contrast.

Earlier work on effective properties for periodic and stationary random media \cite{BergmanC}, \cite{MiltonCL}, and \cite{GoldenPap} show that the effective conductivity for a composite medium is an analytic function of the contrast. The effective conductivity function is seen to be nonzero and  analytic  off the negative real axis and is determined by its singularities and zeros. Estimates for effective properties are obtained from partial knowledge of the singularities and zeros.  The work of \cite{Bruno} develops power series solutions to bound the poles and zeros of the effective conductivity function. This provides bounds on the effective conductivity function for the class of inclusion geometries discussed here.

Subsequent research activity has provided insights on the frequency spectrum for high contrast periodic media. Rigorous analysis establishing existence of  band gaps in the limit of high contrast is developed in \cite{FigKuch2,FigKuch3,FigKuch1}. There an asymptotic analysis is carried out for establishing the existence of band gaps for photonic and acoustic crystals made from thin walled cubic lattices containing a low permittivity material such as air. Band gaps are shown to appear in the limit as walls become thin and the permittivity of the wall increases to infinity. The work \cite{HempelLienau}  considers a high contrast problem but with periodically distributed inclusions embedded inside a host material. The coefficient associated with the partial differential operator inside the host material is very large and the band-gap structure is analyzed in the limit when the coefficient is sent to infinity. Here the geometry is kept fixed for all values of the coefficient and it is shown asymptotically that band gaps emerge as this coefficient approaches infinity.   This phenomena is shown to be generic and holds  for a very general class of inclusion shapes. Asymptotic expansions for Bloch eigenvalues are developed and applied to this setting in \cite{AmmariKang1}. The expansions are in terms of the contrast and developed using a boundary integral perturbation approach based on the generalized Rouch\'e's theorem \cite{AmmariKang1}. The approach provides explicit asymptotic expressions for the band structure in the high contrast limit.

Along related lines the work of \cite{Zhikov}  examines the frequency spectrum of  high contrast periodic media in the high contrast sub-wavelngth limit. For periods of size $d$  the coefficient inside the included phase is proportional to  $d^2$ and   a multi branched quasi-static dispersion relation emerges for the Bloch spectra in the limit $d\rightarrow 0$, see \cite{Zhikov}. This effect is also observed for models of two dimensional electromagnetics  \cite{BouchetteFelbacq} and is responsible for the generation of artificial magnetism. Recent work applies a power series approach to recover this spectra \cite{3P}, as well as dispersion relations associated with wave transport inside plasmonic crystals \cite{3Peffective} and for periodic crystals of micro-resonators \cite{ShipmanMicroresonators}. Power series for the recovery of dispersion relations that rigorously demonstrate backwards wave behavior across selected frequency intervals are developed in \cite{ChenLipton}. In that work the electrostatic spectrum for a three phase problem is used to develop an analytic representation formula for determining the existence or non existence of pass bands. The techniques developed in the power series based approaches listed above are distinct from those developed here and instead work directly with coefficients obtained from formal power series expansions. These approaches apply majoring series techniques to establish convergence of the formal series.

The paper is organized as follows: In the next section we introduce the Hilbert space formulation of the problem and the variational formulation of the quasi-static resonance problem. The completeness of the eigenfunctions associated with the quasi-static spectrum is established and a spectral representation for the operator $-\nabla\cdot (k \chi_{Y \setminus D} + \chi_{D})\nabla$ is obtained. These results are collected and used to continue the frequency band structure into the complex plane, see  Theorem \ref{extension} of section \ref{bandstructure}. Spectral perturbation theory \cite{KatoPerturb} is applied to recover the power series expansion for Bloch spectra in section \ref{asymptotic}. The leading order spectral theory is developed for quasi-periodic $\alpha\not=0$ and periodic $\alpha=0$ problems in sections \ref{limitspectra} and \ref{limitspeczero}.  The main theorems on radius of convergence and separation of spectra given by Theorems \ref{separationandraduus-alphanotzero} and \ref{separationandraduus-alphazero}  are presented in  section \ref{radius}. The class of buffered inclusions is introduced in section \ref{radiusgeneralshape} and the explicit radii of convergence for a random suspension of buffered disks is presented in section \ref{radiusmultiplescatterers}. Explicit formulas for each term of the power series expansion is recovered and expressed in terms of layer potentials in section \ref{leading-order}. In section \ref{explicitfirstorder} the explicit formula for the first order correction in the power series is presented in the form of  the Dirichlet energy of the solution of a transmission boundary value problem. This formula follows from the layer potential representation for the first term and agrees with the first order correction obtained in the work of \cite{AmmariKang1}. The explicit formulas for the convergence radii are derived in section \ref{derivation} as well as hands on proofs of Theorems  \ref{separationandraduus-alphanotzero} and \ref{separationandraduus-alphazero} and the explicit error estimates for  the series truncated after N terms.

\section{Hilbert space setting, quasi-periodic resonances and representation formulas}
\label{layers}

We denote the spaces of all $\alpha$ quasi-periodic complex valued functions belonging to $L_{loc}^2(\mathbb{R}^d)$ by $L_{\#}^2(\alpha,Y)$ and the $L^2$ inner product over $Y$ is written
\begin{eqnarray}
(u,v)=\int_Y u\overline{v}\,dx.
\label{l2}
\end{eqnarray}
For $\alpha\not=0$ the eigenfunctions $h$ for  \eqref{Eigen1} belong to the space
\begin{equation}
H^1_{\#}(\alpha,Y) = \{ h \in  H_{loc}^1(\mathbb{R}^d): \hbox{$h$ is $\alpha$ quasiperiodic} \}.
\label{H1}
\end{equation}
\linebreak
The space  $H^1_{\#}(\alpha,Y)$  is a Hilbert space under the inner product
\begin{equation}
\langle u,v \rangle= \int_{Y} \nabla u(x) \cdot \nabla \bar{v}(x)dx.
\label{innerproduct}
\end{equation}
\linebreak
When $\alpha=0$, the pair $h(x)=1$, $\omega^2=0$ is a solution to \eqref{Eigen1}. For this case the remaining eigenfunctions associated with nonzero eigenvalues are orthogonal to  $1$ in the $L^2(Y)$ inner product.  These eigenfunctions are periodic and belong to  $L_{loc}^2(\mathbb{R}^d)$. The  set of $Y$ periodic functions with zero average over $Y$ belonging to $L_{loc}^2(\mathbb{R}^d)$ is denoted by $L^2_{\#}(0,Y)$.  The periodic eigenfunctions of \eqref{Eigen1} associated with nonzero eigenvalues  belong to the space
\begin{equation}
H^1_{\#}(0,Y) = \{ h \in  H_{loc}^1(\mathbb{R}^d): \hbox{$h$ is periodic, $\int_Y h\,dx=0$} \}.
\label{H1}
\end{equation}
\linebreak
The space  $H^1_{\#}(0,Y)$  is also Hilbert space with the inner product 
$\langle u,v \rangle$ defined by \eqref{innerproduct}.

For any $k \in \mathbb{C}$, the the variational formulation of the eigenvalue problem \eqref{Eigen1} for $h$ and $\omega^2$ is given by
\begin{eqnarray}
B_k(h,v)=\omega^2(h,v)
\label{weak}
\end{eqnarray}
 for all $v$ in $H^1_{\#}(\alpha,Y)$ where $B_k : H^1_{\#}(\alpha,Y) \times H^1_{\#}(\alpha,Y)\longrightarrow \mathbb{C}$ is the sesquilinear form
\begin{equation}
B_k(u,v) =   k\int_{Y \setminus D} \nabla u(x) \cdot \nabla \bar{v}(x)dx +  \int_{D} \nabla u(x) \cdot \nabla \bar{v}(x)dx.
\label{sesquoperator}
\end{equation}
\linebreak
The linear operator $T^\alpha_k: H^1_{\#}(\alpha,Y) \longrightarrow H^1_{\#}(\alpha,Y)$ associated with $B_k$ is defined by
\begin{equation}
\langle T^\alpha_k u, v \rangle := B_k(u,v).
\label{Threetwo}
\end{equation}

In what follows we decompose $H^1_{\#}(\alpha,Y)$ into invariant subspaces of source free modes and identify the associated quasi-periodic resonance spectra. This decomposition will provide an explicit spectral representation for the operator $T^\alpha_k$, see Theorem \ref{T_z spectrum}. We first address the case $\alpha\in Y^\star\setminus \{0\}$. Let $W_1 \subset H^1_{\#}(\alpha,Y)$ be the completion in $H^1_{\#}(\alpha,Y)$ of the subspace of functions with support away from $D$, and let $W_2 \subset H^1_{\#}(\alpha,Y)$ be the subspace of functions in $H^1_0(D)$ extended by zero into $Y$.  Clearly $W_1$ and $W_2$ are orthogonal subspaces of $H^1_{\#}(\alpha,Y)$, so define $W_3 := (W_1\oplus W_2)^{\bot}$.  We therefore have 
\begin{equation}
H^1_{\#}(\alpha,Y) = W_1 \oplus W_2 \oplus W_3.
\end{equation}
The orthogonal decomposition and integration by parts shows that elements $u \in W_3$ are harmonic separately in $D$ and $Y\setminus D$.

Now consider $\alpha=0$ and decompose $H^1_\#(0,Y)$. Let $W_1 \subset H^1_{\#}(0,Y)$ be the completion in $H^1_{\#}(0,Y)$ of the subspace of functions with support away from $D$. Here let $\tilde{H}^1_0(D)$ denote the subspace of functions $H^1_0(D)$ extended by zero into $Y\setminus D$ and let $1_Y$ be the indicator function of $Y$. We define $W_2 \subset H^1_{\#}(0,Y)$ be the subspace of functions given by 
\begin{equation}
\label{definitionW2periodic}
W_2 = \{ u = \tilde{u} - \left( \int_D \tilde{u} dx \right) 1_Y \; \mid \;  \tilde{u} \in \tilde{H}^1_0(D)\}.
\end{equation}
Clearly $W_1$ and $W_2$ are orthogonal subspaces of $H^1_{\#}(0,Y)$, and $W_3 := (W_1\oplus W_2)^{\bot}$.  As before we have 
\begin{equation}
H^1_{\#}(0,Y) = W_1 \oplus W_2 \oplus W_3
\end{equation}
and $W_3$ is identified with the subspace of $H^1_{\#}(0,Y)$ functions that are harmonic inside $D$ and $Y\setminus D$ respectively. The orthogonality between  $W_2$ and $W_3$ follows from the identity $\int_{\partial D}\partial_n w \, ds=0$ for $w\in W_3$.
We summarize with the following observation. 
\begin{lemma}
For every $\alpha\in Y^\star$, if $u \in W_3$ then $u$ is harmonic in $Y\setminus D$ and $D$ separately.
 \label{alphasplit}
 \end{lemma}
To set up the spectral analysis observe that Lemma \ref{alphasplit}, together with uniqueness of traces onto $\partial D$ of functions in $H^1_{\#}(\alpha,Y)$ for $\alpha \in Y^*$, implies that elements of $W_3$ can be represented in terms of single layer potentials supported on $\partial D$. We introduce the  $d$-dimensional $\alpha$-quasi-periodic Green's function, $d=2,3$
given by, see, e.g., \cite{AmmariKangLee},
\begin{equation}
G^{\alpha}(x,y) = -\sum \limits_{n \in \mathbb{Z}^d} \frac{e^{i(2\pi n + \alpha)\cdot(x-y)}}{|2\pi n +\alpha |^2}\hbox{ for $\alpha\not=0$}
\label{Greensalpha}
\end{equation}
and the periodic Green's function given by
\begin{equation}
G^{0}(x,y) = -\sum \limits_{n \in \mathbb{Z}^d\setminus\{0\}} \frac{e^{i2\pi n\cdot(x-y)}}{|2\pi n|^2}\hbox{ for $\alpha=0$}.
\label{Greensalphazero}
\end{equation}

Let $H^{1/2}(\partial D)$ be the fractional Sobelev space on $\partial D$ defined in the usual way, and denote its dual by $(H^{1/2}(\partial D))^* = H^{-1/2}(\partial D)$.  For $\rho \in H^{-1/2}(\partial D)$, and $\alpha\in Y^\star$ define the single layer potential
\begin{equation}
S_D[\rho](x) = \int_{\partial D} G^\alpha(x,y)\rho(y)d\sigma(y)\text{,           }x\in Y.
\end{equation}
It follows from \cite{CostabelBdryOps}, that for any $\rho \in H^{-1/2}(\partial D)$ 
\begin{eqnarray}
\Delta S_D\rho & = & 0\text{ in } D \text{ and } Y\setminus D,\nonumber \\
S_D\rho \mid^{+}_{\partial D} & = & S_D\rho \mid^{-}_{\partial D},\nonumber\\
\frac{\partial}{\partial \nu}S_D\rho \mid^{\pm}_{\partial D} & = & \pm \frac{1}{2}\rho + (\tilde{K}^{-\alpha}_D)^{*}\rho,
\label{singlelayer}
\end{eqnarray}
where $\nu$ is the outward directed normal vector on $\partial D$ and $(\tilde{K}^{-\alpha}_D)^{*}$ is the Neumann Poincar\'e operator defined by
\begin{equation}
(\tilde{K}^{-\alpha}_D)^{\ast}\rho(x) = \text{ p. v. } \int_{\partial D} \frac{\partial G^\alpha(x,y)}{\partial \nu(x)}\rho(y)d\sigma(y)\text{,    } x \in \partial D,
\label{neumannp}
\end{equation}
and
$\tilde{K}^{\alpha}_D$ is the Neumann Poincar\'e operator
\begin{equation}
\tilde{K}^{\alpha}_D\rho(x) = \text{ p. v. } \int_{\partial D} \frac{\partial G^\alpha(y,x)}{\partial \nu(y)}\rho(y)d\sigma(y)\text{,    } x \in \partial D.
\label{adjointneumannp}
\end{equation}

In what follows we assume the boundary  $\partial D$ is $C^{1,\gamma}$, for some $\gamma>0$.
Here the layer potentials  $\tilde{K}^{\alpha}_D$, and $(\tilde{K}^{-\alpha}_D)^*$ are continuous linear mappings from $L^2(\partial D)$ to $L^2(\partial D)$ and compact, since  $\frac{\partial G^\alpha(x,y)}{\partial \nu(x)}$ is a continuous kernel of order $d-2$ in dimensions $d=2,3$.  The operator $S_{D}$ is a continuous linear map from $H^{-1/2}(\partial D)$ into $H^1_\#(\alpha,Y)$ and we define $S_{\partial D} \rho = S_D \rho \mid_{\partial D}$ for all $\rho \in H^{-1/2}(\partial D)$.  Here $S_{\partial D} : H^{-1/2}(\partial D) \longrightarrow H^{1/2}(\partial D)$ is continuous and invertible, see \cite{CostabelBdryOps}.

One readily verifies the symmetry 
\begin{eqnarray}
G^\alpha(x,y)=G^{-\alpha}(y,x),
\label{greensymmetries}
\end{eqnarray}
and application delivers the Plemelj symmetry for $\tilde{K}^{-\alpha}$, $(\tilde{K}^{-\alpha})^*$ and $S_{\partial D}$ as operators on $L^2(\partial D)$  given by
\begin{eqnarray}
\tilde{K}^{-\alpha}S_{\partial D}=S_{\partial D}(\tilde{K}^{-\alpha})^*.
\label{Plemelj}
\end{eqnarray}
Moreover as seen in \cite{Shapero} the operator $-S_{\partial D}$ is positive and selfadjoint in $L^2(\partial D)$ and in view of \eqref {Plemelj}    $(\tilde{K}^{-\alpha}_D)^*$ is  a compact operator on  $H^{-1/2}(\partial D)$.

Let $G : W_3 \longrightarrow H^{1/2}(\partial D)$ be the trace operator, which is bounded and onto.
\begin{lemma}
$S_D: H^{-1/2}(\partial D) \longrightarrow W_3$ is a one-to-one, bounded linear map with bounded inverse $S_D^{-1} = S_{\partial D}^{-1}G$.
\end{lemma}
\begin{proof}
Let $\rho \in H^{-1/2}(\partial D)$, and set $f = S_D \rho$. Then by the first equation of \eqref{singlelayer}, $f$ is harmonic in $D$ and $Y\setminus D$ separately, and so for any $v_1 \in W_1$ and $v_2 \in W_2$ we have
\begin{equation}
\langle f,v_1\rangle = 0 = \langle f, v_2 \rangle\text{.}
\end{equation}
But $W_3 = (W_1\oplus W_2)^{\bot}$, so $f = S_D\rho \in W_3$ for every $\rho \in H^{-1/2}(\partial D)$. 
\par
Now suppose $u\in W_3$, and consider $G u = u\mid_{\partial D} \in H^{1/2}(\partial D)$.  For all $x \in Y$ define $w(x) = S_D(S_{\partial D}^{-1}G u)$.  Since $u, w \in W_3$, it follows that $w-u \in W_3$ as well.  Since $G u = G w$, we have that $G (w-u) = 0$, and so $w-u \in (W_1 \oplus W_2)$. But $W_3 = (W_1\oplus W_2)^{\bot}$, so $w=u$ as desired.
\end{proof}
We introduce an auxiliary operator $T:W_3 \longrightarrow W_3$, given by the sesquilinear form
\begin{equation}
\langle Tu,v \rangle = \frac{1}{2} \int_{Y \setminus D} \nabla u(x) \cdot \nabla \bar{v}(x)dx - \frac{1}{2}\int_{D} \nabla u(x) \cdot \nabla \bar{v}(x)dx.
\label{threefourteen}
\end{equation}
The next theorem will be useful for the spectral decomposition of $T^\alpha_k$ and in the proof of Theorem \ref{reptheorem1}.
\begin{theorem}
The linear map $T$ defined in equation \eqref{threefourteen} is given by
$$T = S_D (\tilde{K}^{-\alpha}_D)^*S_D^{-1}$$
and is compact and self-adjoint.
\end{theorem}
\begin{proof}
For $u,v \in W_3$, consider
\begin{equation}
\langle S_D (\tilde{K}^{-\alpha}_D)^*S_D^{-1}u, v \rangle = \int_Y \nabla [S_D (\tilde{K}^{-\alpha}_D)^*S_D^{-1}u] \cdot \nabla \bar{v}.
\end{equation}
Since $\Delta S_D\rho = 0$ in $D$ and $Y\setminus D$ for any $\rho \in H^{-1/2}(\partial D)$, an integration by parts yields
$$\langle S_D (\tilde{K}^{-\alpha}_D)^*S_D^{-1}u, v \rangle = \int_{\partial D} \bar{v} ( \frac{\partial [S_D (\tilde{K}^{-\alpha}_D)^*S_D^{-1}u]}{\partial \nu}\mid^-_{\partial D} -  \frac{\partial [S_D (\tilde{K}^{-\alpha}_D)^*S_D^{-1}u]}{\partial \nu}\mid^+_{\partial D})d\sigma.$$
Applying the jump conditions from \eqref{singlelayer} yields
\begin{equation}
\langle S_D (\tilde{K}^{-\alpha}_D)^*S_D^{-1}u, v \rangle = -\int_{\partial D} (\tilde{K}^{-\alpha}_D)^*S_D^{-1}u\bar{v}d\sigma.
\label{Tboundary}
\end{equation}
Note that by the same jump conditions
\begin{equation}
\label{samejump}
(\tilde{K}^{-\alpha}_D)^*S_D^{-1}u = \frac{1}{2}(\frac{\partial u}{\partial \nu}|^-_{\partial D} + \frac{\partial u}{\partial \nu}|^+_{\partial D}).
\end{equation}
Application of \eqref{samejump} to equation \eqref{Tboundary} and an integration by parts yields the desired result. Compactness follows directly from the properties of 
$S_D$ and $(\tilde{K}^{-\alpha})^\ast$.
\end{proof}

Rearranging terms in the weak formulation of \eqref{sourcefree} and writing $\mu=1/2-\lambda$ delivers the equivalent eigenvalue problem for quasi-periodic electrostatic resonances.

$$\langle Tu,v \rangle = \mu\langle u,v \rangle\text{,  } u,v \in W_3.$$
Since $T$ is compact and self adjoint on $W_3$, there exists a countable subset $\{ \mu_i \}_{i\in \mathbb{N}}$ of the real line with a single accumulation point at $0$ and an associated family of orthogonal finite-dimensional projections $\{ P_{\mu_i}\}_{i\in \mathbb{N}}$ such that
$$\langle \sum_{i=1}^\infty P_{\mu_i}u, v \rangle = \langle u,v \rangle \text{,       } u,v \in W_3$$
and
$$\langle \sum_{i=1}^\infty \mu_iP_{\mu_i}u, v \rangle = \langle Tu,v \rangle \text{,       } u,v \in W_3.$$
Moreover, it is clear by \eqref{threefourteen} that 
$$-\frac{1}{2} \leq \mu_i \leq \frac{1}{2}.$$
The upper bound $1/2$ is the eigenvalue associated with the eigenfunction $\Pi\in H_{\#}^1(\alpha,Y)$ such that $\Pi=1$ in $D$ and is harmonic on $Y\setminus D$.  In section \ref{radiusgeneralshape} an explicit lower bound $\mu^-$ is identified such that the inequality $-1/2<\mu^-\leq \mu_i$,   holds for a generic class of geometries uniformly with respect to $\alpha\in Y^\star$.
\begin{lemma}
The eigenvalues $\{\mu_i\}_{i\in\mathbb{N}}$ of $T$ are precisely the eigenvalues of the Neumann-Poincar\'e operator  $(\tilde{K}^{-\alpha}_D)^*$ associated with quasi-periodic double layer potential restricted to $\partial D$. \end{lemma}
\begin{proof}
If a pair $(\mu, u)$ belonging to $(-1/2, 1/2] \times W_3$ satisfies $Tu = \mu u$ then
$S_D(\tilde{K}_D^{-\alpha})^\ast S_D^{-1}u= \mu u$. Multiplication of both sides by $S_D^{-1}$ shows that $S_D^{-1}u$ is an eigenfunction for  
function for $(\tilde{K}_D^{-\alpha})^\ast$  associated with $\mu$. Suppose the pair $(\mu,w)$ belongs to $(-1/2,1/2]\times
H^{-1/2}(\partial D)$ and satisfies $(\tilde{K}_D^{-\alpha})^\ast w=\mu w$. Since the trace map from $W_3$ to $H^{1/2}(\partial D)$ is onto then there is a  $u$ in $W_3$ for which $w=S_D^{-1}u$ and  $(\tilde{K}_D^{-\alpha})^\ast S_D^{-1}u=\mu S_D^{-1}u.$
Multiplication of this identity by $S_D$ shows that $u$ is an eigenfunction for $T$ associated with $\mu$.
\end{proof}

Finally, we see that if $u_1\in W_1$ and $u_2\in W_2$, then  

$$\langle Tu_1,v \rangle = \frac{1}{2}\langle u_1,v \rangle \text{,}$$
$$\langle Tu_2, v\rangle = -\frac{1}{2}\langle u_2,v\rangle $$
for all $v\in H^1_{\#}(\alpha,Y)$.

Let $Q_1, Q_2$ be the orthogonal projections of $H^1_{\#}(\alpha,Y)$ onto $W_1$ and $W_2$ respectively, and define $P_1 := Q_1 + P_{1/2} \text{,  } P_2:= Q_2 $.  Here $P_{1/2}$ is the projection onto the one dimensional subspace spanned by the function $\Pi\in H_{\#}^1(\alpha,Y)$. Then $\{ P_1, P_2\} \cup \{ P_{\mu_i}\}_{-\frac{1}{2} < \mu_i < \frac{1}{2}}$ is an orthogonal family of projections, and
$$\langle P_1u + P_2u + \sum_{-\frac{1}{2} < \mu_i < \frac{1}{2}} P_{\mu_i}u, v\rangle = \langle u,v\rangle$$
for all $u,v \in H^1_{\#}(\alpha,Y)$.

We now recover the spectral decomposition for $T^\alpha_k$ associated with the sesqualinear form \eqref{Threetwo}.
\begin{theorem}
\label{T_z spectrum}
The linear operator $T^\alpha_k: H^1_{\#}(\alpha,Y) \longrightarrow H^1_{\#}(\alpha,Y)$ associated with the sesqualinear form $B_k$ is  is given by
$$\langle T^\alpha_k u,v \rangle = \langle k P_1u + P_2u + \sum_{-\frac{1}{2} < \mu_i < \frac{1}{2}}[k(1/2 + \mu_i) + (1/2-\mu_i)] P_{\mu_i}u, v\rangle$$
for all $u,v\in H^1_{\#}(\alpha,Y)$.
\end{theorem}
\begin{proof}
For $u, v\in H^1_{\#}(\alpha,Y)$ we have
$$B_k(P_{\mu_i}u, v) = k\int_{Y \setminus D} \nabla P_{\mu_i}u \cdot \nabla \bar{v} +  \int_{D} \nabla P_{\mu_i}u \cdot \nabla \bar{v}.$$
Since $P_{\mu_i}u$ is an eigenvector corresponding to $\mu_i \neq \pm \frac{1}{2}$, we have
$$\int_{Y \setminus D} \nabla P_{\mu_i}u \cdot \nabla \bar{v} = \frac{(1/2+\mu_i)}{(1/2-\mu_i)}\int_{D} \nabla P_{\mu_i}u \cdot \nabla \bar{v}$$
and so we calculate
$$B_k(P_{\mu_i}u, v) = [k\frac{(1/2+\mu_i)}{(1/2-\mu_i)} +1]\int_{D} \nabla P_{\mu_i}u \cdot \nabla \bar{v}.$$
But we also know that
$$\int_D \nabla P_{\mu_i} u \cdot \nabla \bar{v} = (1/2-\mu_i)\int_Y \nabla P_{\mu_i} u \cdot \nabla \bar{v}$$
and so
$$B_k(P_{\mu_i}u, v) = [k(1/2 + \mu_i) + (1/2-\mu_i)] \int _Y \nabla P_{\mu_i} u \cdot \nabla \bar{v}.$$
Since we clearly have
$$B_k(P_1u, v) = k \int_{Y \setminus D} \nabla P_1u \cdot \nabla \bar{v}\text{,}$$
$$B_k(P_2u, v) = \int_{D} \nabla P_2u \cdot \nabla \bar{v}\text{,}$$
and the projections $P_1, P_2, P_{\mu_i}$ are mutually orthogonal for all $-\frac{1}{2} < \mu_i < \frac{1}{2}$, the proof is complete.
\end{proof}

It is evident that $T^\alpha_k: H^1_{\#}(\alpha,Y) \longrightarrow H^1_{\#}(\alpha,Y)$ is invertible whenever 
\begin{eqnarray} k \in \mathbb{C}\setminus Z \hbox{ where } Z=\{ \frac{\mu_i - 1/2}{\mu_i + 1/2} \}_{\{-\frac{1}{2} \leq \mu_i \leq \frac{1}{2}\}}
\label{invertability}
\end{eqnarray}
 and for $z=k^{-1}$,
\begin{eqnarray}
(T^\alpha_k)^{-1}= z P_1u + P_2u + \sum_{-\frac{1}{2} < \mu_i < \frac{1}{2}}z[(1/2 + \mu_i) + z(1/2-\mu_i)] ^{-1}P_{\mu_i}.
\label{inverse}
\end{eqnarray}
For future reference we also introduce the set $S$ of $z\in\mathbb{C}$ for which $T^\alpha_k$ is not invertible given by
\begin{eqnarray} 
S=\{ \frac{\mu_i + 1/2}{\mu_i - 1/2} \}_{\{-\frac{1}{2} < \mu_i < \frac{1}{2}\}}
\label{noninvertability}
\end{eqnarray}
which also lies on the negative real axis. In section \ref{radiusgeneralshape} we will provide explicit upper bounds on $S$ that depend upon the geometry of the inclusions.

Collecting results, the spectral representation of the operator $-\nabla\cdot(k \chi_{Y \setminus D} + \chi_{D})\nabla$ on $H^1_{\#}(\alpha,Y)$ is given by
 \begin{eqnarray}
 -\nabla\cdot (k \chi_{Y \setminus D} + \chi_{D})\nabla=-\Delta_\alpha T_k^\alpha,
 \label{representation}
 \end{eqnarray}
in the sense of linear functionals over the space $H^1_{\#}(\alpha,Y)$. Here $-\Delta_\alpha$ is the Laplace operator associated with the bilinear form $\langle\cdot,\cdot\rangle$ defined on ${H^1_{\#}(\alpha,Y)}$. This formulation is useful since it separates the effect of the contrast $k$ from the underlying geometry of the crystal.

\section{Band Structure for Complex Coupling Constant }
\label{bandstructure}

We set $\omega^2=\lambda$ in  \eqref{Eigen1} and extend the Bloch eigenvalue problem to complex coefficients $k$ outside the set $Z$ given by \eqref{invertability}. The operator representation is applied to write the Bloch eigenvalue problem as

\begin{eqnarray}
 -\nabla\cdot (k \chi_{Y \setminus D} + \chi_{D})\nabla u=-\Delta_\alpha T_k^\alpha u=\lambda u.
 \label{representationform}
 \end{eqnarray}
 
We characterize the Bloch spectra by analyzing the operator 
\begin{eqnarray}
B^\alpha(k)=(T_k^\alpha)^{-1}(-\Delta_{\alpha})^{-1},
\label{inverse}
\end{eqnarray}
where the operator $(-\Delta_{\alpha})^{-1}$ defined for all $\alpha\in Y^\ast$ is given by
\begin{eqnarray}
(-\Delta_{\alpha})^{-1} u(x)=-\int_Y G^\alpha(x,y) u(y)\,dy.
\label{inverselaplacian}
\end{eqnarray}

The operator $B^\alpha(k): L^2_{\#}(\alpha,Y) \longrightarrow H^1_{\#}(\alpha,Y)$ is easily seen to be bounded for $k\not\in Z$, see Theorem \ref{bounded}. Since $H^1_{\#}(\alpha,Y)$ embeds compactly into $L^2_{\#}(\alpha,Y)$ we find by virtue of Poincare's inequality that $B^\alpha(k)$ is a bounded compact linear operator on $L^2_\#(\alpha,Y)$ and therefore has a discrete spectrum $\{ \gamma_i(k,\alpha) \}_{i \in \mathbb{N}}$ with a possible accumulation point at $0$, see Remark \ref{compact2}. The corresponding  eigenspaces are finite dimensional and the eigenfunctions $p_i\in L^2_{\#}(\alpha,Y)$ satisfy
\begin{eqnarray}
B^\alpha(k)p_i(x)=\gamma_i(k,\alpha)p_i(x)\hbox{ for $x$ in $Y$}
\label{compactproblem}
\end{eqnarray}
and also belong to $H^1_\#(\alpha,Y)$. Note further for $\gamma_i\not =0$ that \eqref{compactproblem} holds if and only if \eqref{representationform} holds with $\lambda_i(k,\alpha)=\gamma_i^{-1}(k,\alpha)$, and $-\Delta_\alpha T_k^\alpha p_i=\lambda_i(k,\alpha) p_i.$
Collecting results we have the following theorem
\begin{theorem}
\label{extension}
Let $Z$ denote the set of points on the negative real axis defined by \eqref{invertability}. Then the Bloch eigenvalue problem \eqref{Eigen1} for the operator $-\nabla(k\chi_{Y\setminus D}+\chi_D)\nabla$ associated with the sesquilinear form \eqref{sesquoperator} can be extended for values of the coupling constant $k$ off the positive real axis into $\mathbb{C}\setminus Z$, i.e., for each $\alpha\in Y^\star$  the Bloch eigenvalues are of finite multiplicity and denoted by $\lambda_j(k,\alpha)=\gamma_j^{-1}(k,\alpha)$, $j\in \mathbb{N}$ and the band structure
\begin{equation}
\lambda_j(k,\alpha)=\omega^2,\hbox{ $j\in\mathbb{N}$}
\label{DispersionRelations}
\end{equation}
extends to complex coupling constants $k\in\mathbb{C}\setminus Z$.
\end{theorem}

\section{Power Series Representation of Bloch Eigenvalues for High Contrast Periodic Media }
\label{asymptotic}
In what follows we set $\gamma=\lambda^{-1}(k,\alpha)$ and analyze the spectral problem
\begin{equation}
B^\alpha(k) u = \gamma (k,\alpha) u
\label{forpointtwo}
\end{equation}
Henceforth we will analyze the high contrast limit by  by developing a power series in $z=\frac{1}{k}$ about $z=0$ for the spectrum of the family of operators associated with \eqref{forpointtwo}. 
$$\begin{array}{lcl}
B^{\alpha}(k) & := & (T_k^\alpha)^{-1}(-\Delta_{\alpha})^{-1}\\
& = &  (zP_1 + P_2 + z\sum_{-\frac{1}{2} < \mu_i < \frac{1}{2}}[(1/2 + \mu_i) + z(1/2-\mu_i)]^{-1} P_{\mu_i})(-\Delta_{\alpha})^{-1}\\
&=& A^\alpha(z).
\end{array}$$
Here we define the operator $A^\alpha(z)$ such that $A^\alpha(1/k)=B^\alpha(k)$ and the associated eigenvalues $\beta(1/k,\alpha)=\gamma(k,\alpha)$ and the spectral problem is $A^\alpha(z)u=\beta(z,\alpha)u$ for $u\in L^2_{\#}(\alpha,Y)$.

It is easily seen from the above representation that  $A^{\alpha}(z)$ is self-adjoint for $k\in\mathbb{R}$ and is a family of bounded operators taking $L^2_{\#}(\alpha,Y)$ into itself and we have the following:
\begin{lemma}
\label{inverseoperator}
$A^{\alpha}(z)$ is holomorphic on $\Omega_0 := \mathbb{C} \setminus S$. 
Where $S=\cup_{i\in\mathbb{N}} z_i$ is the collection of points $z_i=(\mu_i+1/2)/(\mu_i-1/2)$ on the negative real axis associated with
the eigenvalues $\{\mu_i\}_{i\in\mathbb{N}}$. The set $S$ consists of poles  of $A^\alpha(z)$ with only one accumulation point at $z=-1$.
\end{lemma}

In the sections \ref{radiusgeneralshape}   and \ref{radiusmultiplescatterers} we develop explicit lower bounds  $-1/2<\mu^-\leq \mu^-(\alpha)=\min_i\{\mu_i\}$, that hold for generic classes of inclusion domains $D$ and for every $\alpha\in Y^\star$. The corresponding upper bound $z^+$ on $S$ is written
\begin{eqnarray}
 \max_i \{z_i\}=\frac{\mu^-(\alpha)+1/2}{\mu^-(\alpha)-1/2}=z^\ast\leq z^+<0.
\label{bdsonS}
\end{eqnarray}

Let $\beta^\alpha_0 \in \sigma(A^{\alpha}(0))$ with spectral projection $P(0)$, and let $\Gamma$ be a closed contour in $\mathbb{C}$ enclosing $\beta^\alpha_0$ but no other $\beta \in \sigma(A^{\alpha}(0))$.
The spectral projection associated with $\beta^\alpha (z) \in \sigma(A^{\alpha}(z))$ for $\beta^\alpha(z) \in \text{int}(\Gamma)$ is denoted by $P(z)$. We write $M(z) = P(z)L^2_{\#}(\alpha,Y)$ and suppose for the moment that $\Gamma$ lies in the resolvent of $A^\alpha(z)$ and $\text{dim}(M(0)) =\text{dim}(M(z))= m$, noting that  Theorems \ref{separationandraduus-alphanotzero} and \ref{separationandraduus-alphazero}  provide explicit conditions for when this holds true.  Now define $\hat{\beta}^\alpha(z) = \frac{1}{m} \text{tr}(A^{\alpha}(z)P(z))$, the weighted mean of the eigenvalue group $\{ \beta^\alpha_1(z), \ldots \beta^\alpha_m(z) \}$ corresponding to $\beta^\alpha_1(0) = \ldots = \beta^\alpha_m(0) = \beta^\alpha_0$.   We  write the weighted mean as
\begin{equation}
\label{analyticforband}
\hat{\beta}^\alpha(z) = \beta^\alpha_0 +  \frac{1}{m} \text{tr}[(A^{\alpha}(z)-\beta^\alpha_0)P(z)].
\end{equation}
Since $A^{\alpha}(z)$ is analytic in a neighborhood of the origin we  write
\begin{equation}
A^{\alpha}(z) = A^{\alpha}(0) + \sum_{n=1}^{\infty} z^nA^{\alpha}_n.
\end{equation}
The explicit form of the sequence $\{A^{\alpha}_n\}_{n\in \mathbb{N}}$ is given in Section \ref{radius}. Define the resolvent of $A^{\alpha}(z)$ by
$$R(\zeta, z) = (A^{\alpha}(z) - \zeta )^{-1}\text{,}$$
and expanding successively in Neumann series and power series we have the identity

\begin{equation}
\begin{array}{lcl}
R(\zeta,z) & = & R(\zeta,0)[I + (A^{\alpha}(z) - A^{\alpha}(0))R(\zeta,0)]^{-1} \\
\\
& = & R(\zeta,0)+\sum_{p=1}^\infty [-(A^{\alpha}(z) - A^{\alpha}(0))R(\zeta,0)]^p\\
\\
& = & R(\zeta,0) + \sum_{n=1}^{\infty} z^nR_n(\zeta)\text{,}
\end{array}
\label{foursix}
\end{equation}

where
$$R_n(\zeta) = \sum_{k_1 + \ldots k_p = n, k_j \geq 1} (-1)^pR(\zeta,0)A^{\alpha}_{k_1}R(\zeta,0)A^{\alpha}_{k_2}\ldots R(\zeta,0)A^{\alpha}_{k_p}$$
for $n\geq 1$.

Application of  the contour integral formula for spectral projections \cite{SzNagy}, \cite{TKato1}, \cite{TKato2} delivers the expansion for the spectral projection
\begin{equation}
\begin{array}{lcl}
P(z) & = & -\frac{1}{2\pi i} \oint_{\Gamma} R(\zeta, z)d\zeta\\
\\
& = & P(0) + \sum_{n=1}^\infty z^n P_n
\end{array}
\label{Project1}
\end{equation}
where $P_n =  -\frac{1}{2\pi i} \oint_{\Gamma} R_n(\zeta)d\zeta$.  Now we develop the series for the  weighted mean of the eigenvalue group. Start with
\begin{eqnarray}
(A^{\alpha}(z) - \beta^\alpha_0)R(\zeta,z) = I + (\zeta - \beta^\alpha_0)R(\zeta,z)
\label{project2}
\end{eqnarray}
and  we have

\begin{equation}
(A^{\alpha}(z) - \beta^\alpha_0)P(z) = - \frac{1}{2 \pi i}\oint_{\Gamma} (\zeta - \beta^\alpha_0)R(\zeta,z)d\zeta \text{,}
\end{equation}
so
\begin{equation}
\hat{\beta}(z) - \beta^\alpha_0 = - \frac{1}{2m \pi i} \text{tr}\oint_{\Gamma} (\zeta - \beta^\alpha_0)R(\zeta,z)d\zeta.
\label{fourten}
\end{equation}
Equation \eqref{fourten} delivers an analytic representation formula for a Bloch eigenvalue or more generally the eigenvalue group when $\beta^\alpha_0 $ is not a simple eigenvalue.
Substituting the third line of \eqref{foursix} into \eqref{fourten} and manipulation  yields

\begin{equation}
\hat{\beta}^\alpha(z) = \beta^\alpha_0 + \sum_{n=1}^\infty z^n\beta^\alpha_n,
\label{foureleven}
\end{equation}
where
\begin{equation}
{\beta}^\alpha_n = - \frac{1}{2m \pi i} \text{tr}\sum_{k_1+\cdots+k_p=n} \frac{(-1)^p}{p}\oint_{\Gamma}A^{\alpha}_{k_1}R(\zeta,0)A^{\alpha}_{k_2}\ldots R(\zeta,0)A^{\alpha}_{k_p}R(\zeta,0)d\zeta;\hbox{ $n\geq 1$}.
\label{fourtwelve}
\end{equation}

\section{Spectrum in the High Contrast Limit: Quasi-periodic Case}
\label{limitspectra}
We now identify the spectrum of the limiting operator $A^{\alpha}(0)$ when $\alpha \neq 0$.  Using the representation
\begin{equation}
A^{\alpha}(z) = (zP_1 + P_2 + z\sum \limits_{-\frac{1}{2} < \mu_i < \frac{1}{2}} [(1/2 + \mu_i) + z(1/2-\mu_i)]P_{\mu_i})(-\Delta_{\alpha})^{-1} \text{,}
\end{equation}
we see that
\begin{equation}
A^{\alpha}(0) = P_2(-\Delta_{\alpha})^{-1}.
\end{equation}
Denote the spectrum of $A^{\alpha}(0)$ by $\sigma(A^{\alpha}(0))$. The following theorem provides the explicit characterization of $\sigma(A^{\alpha}(0))$. 

\begin{theorem}
\label{equiv1}
Let $-\Delta_D$ be the negative Laplacian with zero Dirichlet boundary conditions on $\partial D$ with inverse $-\Delta^{-1}_D:\,L^2(D)\rightarrow L^2(D)$.  Denote the spectrum of $-\Delta^{-1}_D$ by $\sigma(-\Delta^{-1}_D)$.  Then $\sigma(A^{\alpha}(0)) = \sigma(-\Delta^{-1}_D)$.
\end{theorem}

To  establish the theorem we first  show that the eigenvalue problem
$$P_2(-\Delta_{\alpha})^{-1}u=\lambda u$$ with $\lambda \in \sigma(A^{\alpha}(0))$ and eigenfunction $u\in L^2_\#(\alpha,Y)$ is  equivalent to finding $\lambda$ and $u\in W_2$ for which
\begin{equation}
(u, v)_{L^2(Y)} = \lambda \langle u, v \rangle, \hbox{ for all $v\in W_2$}.
\label{eqiveigenH01}
\end{equation}
To conclude we  will then show that the set of eigenvalues for \eqref{eqiveigenH01} is given by $\sigma(-\Delta^{-1}_D)$. 
To see the equivalence note that we have $u=P_2u$ and for $v \in H^1_{\#}(\alpha,Y)$,
\begin{equation}
\label{fivefour}
\begin{array}{lcl}
\langle P_2(-\Delta_{\alpha})^{-1}u, v \rangle =  \lambda \langle u, v \rangle=\lambda \langle P_2u, v \rangle \end{array}
\end{equation}
hence
\begin{equation}
\label{fivefour2}
\begin{array}{lcl}
\langle (-\Delta_{\alpha})^{-1}u, P_2v\rangle= \lambda \langle u, P_2v \rangle.
\end{array}
\end{equation}
Since $\langle (-\Delta_{\alpha})^{-1}u, v\rangle = \int_Y u\overline{v}\,dx=(u, v)_{L^2(Y)}$ for any $u\in L^2_\#(\alpha,Y)$ and $v \in H^1_{\#}(\alpha,Y)$, equation  \eqref{fivefour2} becomes
\begin{equation}
\label{projectedalpnanotzero}
(u, P_2v)_{L^2(Y)} = \lambda \langle u, P_2v \rangle,
\end{equation}
and the equivalence follows noting that $P_2$ is the  projection of $H^1_{\#}(\alpha,Y)$ onto $W_2$. 

To conclude we  show that the set of eigenvalues for \eqref{eqiveigenH01} is given by $\sigma(-\Delta^{-1}_D)$. Note that $P_2v$ is supported in $D$, so
\begin{equation}
\lambda^{-1} \int \limits_{D} u\overline{P_2v} = \int \limits_{D} \nabla u \cdot \nabla \overline{P_2v}.
\end{equation}
Now since $P_2: H^1_{\#}(\alpha,Y) \rightarrow W_2=\tilde{H}_0^1(D)$ is onto, it follows that $\lambda^{-1}$ is a Dirichlet eigenvalue of the negative Laplacian acting on $D$ and the proof of Theorem \ref{equiv1} is complete.

\section{Spectrum in the High Contrast Limit: Periodic Case}
\label{limitspeczero}

Recall for the periodic case $P_2$ is the projection onto  $W _2$ given by \eqref{definitionW2periodic} and the limiting operator $A^0(0)$ is written
\begin{equation}
A^0(0) = P_2(-\Delta_0)^{-1}.
\end{equation}
Here the operator $(-\Delta_0)^{-1}$ is  compact and self-adjoint on $L^2_{\#}(0,Y)$ and given by
\begin{equation}
(-\Delta_0)^{-1}u(x) = -\int \limits_Y G^0(x,y)u(y) dy.
\end{equation}
Denote the spectrum of $A^{0}(0)$ by $\sigma(A^{0}(0))$. To characterize this spectrum we introduce the sequence of numbers $\{\nu_j\}_{j\in\mathbb{N}}$ given by the positive roots $\nu$ of the  spectral function $S(\nu)$ defined by
\begin{equation}
S(\nu)=\nu\sum \limits_{i\in \mathbb{N}} \frac{ a^2_i}{\nu-\delta^*_i}-1,
\label{roots}
\end{equation}
where $\{\delta^*_j\}_{j\in\mathbb{N}}$ are the Dirichlet eigenvalues for $-\Delta_D$ associated with eigenfunctions $\psi_j$ for which $\int_D\psi_j\,dx\not=0$ and $a_j=|\int_D\psi_j\,dx|$. 
The following theorem provides the explicit characterization of $\sigma(A^{\alpha}(0))$. 
\begin{theorem}
\label{equiv2}
Let $\{\delta'_j\}_{j\in\mathbb{N}}$ denote the collection of Dirichlet eigenvalues  for $-\Delta_D$ associated with eigenfunctions $\psi_j$ for which $\int_D\psi_j\,dx=0$. 
Then $\sigma(A^{0}(0)) = \{{\delta'_j}^{-1}\}_{j\in\mathbb{N}}\cup\{{\nu_j}^{-1}\}_{j\in\mathbb{N}}$.
\end{theorem}

To  establish the theorem we proceed as before  to see that the eigenvalue problem $$P_2(-\Delta_0)^{-1}u=\lambda u$$ with $\lambda \in \sigma(A^{0}(0))$ and eigenfunction $u\in L^2_\#(0,Y)$ is  equivalent to finding $\lambda$ and $u\in W_2$ for which
\begin{equation}
(u, v)_{L^2(Y)} = \lambda \langle u, v \rangle, \hbox{ for all $v\in W_2$}.
\label{eqiveigenW2}
\end{equation}

To conclude we  show that the set of eigenvalues for \eqref{eqiveigenW2} is given by 
$\{{\delta'_j}^{-1}\}_{j\in\mathbb{N}}\cup\{\nu_j^{-1}\}_{j\in\mathbb{N}}$.
We see that $u\in W_2$ and from \eqref{definitionW2periodic} we have the dichotomy: $\int_D \tilde{u}dx=0$ and $u=\tilde{u}\in \tilde{H}^1_0(D)$ or  $\int_D \tilde{u}dx\not=0$ and $u=\tilde{u}-\gamma 1_Y$ with $\gamma=\int_D\tilde{u} dx$.  It is evident for the first case that the eigenfunction $u\in\tilde{H}^1_0(D)$ and for $v\in W_2$ given by 
\begin{equation}
\label{definitionW2periodicdefine}
v = \tilde{v} - \left( \int_D \tilde{v} dx \right) 1_Y \; \hbox{for} \;  \tilde{v} \in \tilde{H}^1_0(D)
\end{equation}
the problem \eqref{eqiveigenW2} becomes
\begin{equation}
\int_Du\overline{\tilde{v}} = \lambda \int_D\nabla u\cdot\nabla\overline{\tilde{v}}, \hbox{ for all $\tilde{v}\in \tilde{H}^1_0(D)$},
\label{eqiveigenW2first}
\end{equation}
and we conclude that $\tilde{u}$ is a Dirichlet eigenfunction with zero average over $D$ so $\lambda\in\{{\delta'_j}^{-1}\}_{j\in\mathbb{N}}$. While for the second, we have $u\in W_2$ and again
\begin{equation}
\int_Du\overline{\tilde{v}} = \lambda \int_D\nabla u\cdot\nabla\overline{\tilde{v}}, \hbox{ for all $\tilde{v}\in \tilde{H}^1_0(D)$}.
\label{eqiveigenW2nd}
\end{equation}
Writing $u=\tilde{u}-\gamma 1_Y$ and integration by parts in \eqref{eqiveigenW2nd} shows that $\tilde{u}\in\tilde{H}^1_0(D)$ is the solution of
\begin{eqnarray}
\Delta \tilde{u}+\nu \tilde{u}=-\nu\gamma \hbox{ for $x\in D$}.
\label{eigenw2}
\end{eqnarray}
We normalize  $\tilde{u}$ so that $\gamma = \int_D\tilde{u}dx=1$  and write
\begin{equation}
\label{series}
\tilde{u} = \sum \limits_{j=1}^{\infty} c_j \psi_j
\end{equation}
where, $\psi_j$ are the Dirichlet eigenfunctions of $-\Delta_D$ associated with eigenvalue $\delta_j$ extended by zero to $Y$.  Substitution of \eqref{series} into \eqref{eigenw2} gives
\begin{equation}
\label{fourierequation}
\sum \limits_{j=1}^{\infty} (-\delta_j + \nu) c_j \psi_j = -\nu.
\end{equation}
Multiplying both sides of \eqref{fourierequation} by $\overline{\psi_k}$ over $Y$ and orthonormality of $\{\psi_j\}_{j \in \mathbb{N}}$, shows that $\tilde{u}$ is given by
\begin{equation}
\label{finalformula}
\tilde{u} = \nu \sum \limits_{k\in \mathbb{N}} \frac{ \int_D \overline{{\psi}_k}}{\nu-\delta^*_k} \psi_k\text{,}
\end{equation}
where $\delta_k^*$ correspond to Dirichlet eigenvalues associated with eigenfunctions for which $\int_D\,\psi_k\,dx\not=0$. 
To calculate $\nu$, we integrate both sides of \eqref{finalformula} over $D$ to recover the identity
\begin{equation}
\nu \sum \limits_{k\in \mathbb{N}} \frac{ a^2_k}{\nu-\delta^*_k}-1=0.
\label{theIdentity}
\end{equation}
It follows from  \eqref{theIdentity} that $\lambda\in\{\nu_i^{-1}\}_{i\in\mathbb{N}}$ and the proof of Theorem \ref{equiv2} is complete.

\section{Radius of Convergence and Separation of Spectra}
\label{radius}

Fix an inclusion geometry specified by the domain $D$. Suppose first $\alpha\in Y^\star$ and $\alpha\not =0$. Recall from Theorem \ref{equiv1}  that the spectrum of $A^\alpha(0)$ is $\sigma(-\Delta_D^{-1})$. Take $\Gamma$ to be a closed contour in $\mathbb{C}$ containing an eigenvalue  $\beta^\alpha_j(0)$ in $\sigma(-\Delta^{-1}_{D})$ but no other element of $\sigma(-\Delta^{-1}_{D})$, see Figure \ref{spectrum}. Define $d$ to be the distance between $\Gamma$ and $ \sigma(-\Delta^{-1}_{D})$, i.e., 
\begin{eqnarray}
d={\rm{dist}}(\Gamma,\sigma(-\Delta^{-1}_{D})=\inf_{\zeta\in\Gamma}\{{\rm{dist}}(\zeta,\sigma(-\Delta^{-1}_{D})\}.
\label{dist}
\end{eqnarray}
The component of the spectrum of $A^\alpha(0)$ inside $\Gamma$  is precisely $\beta^\alpha_j(0)$ and we denote this   by $\Sigma'(0)$. The part of the spectrum of $A^\alpha(0)$ in the  domain exterior to $\Gamma$ is denoted by $\Sigma''(0)$ and $\Sigma''(0)=\sigma(-\Delta^{-1}_{D})\setminus \beta^\alpha_j(0)$. The invariant subspace of $A^\alpha(0)$ associated with $\Sigma'(0)$ is denoted by $M'(0)$ with $M'(0)=P(0)L^2_{\#}(\alpha,Y)$ .

\begin{figure} 
\centering
\begin{tikzpicture}[xscale=0.70,yscale=0.70]
\draw [-,thick] (-6,0) -- (6,0);
\draw [<->,thick] (0,0) -- (1.5,0);
\node [above] at (1,0) {$d$};
\draw [<->,thick] (1.5,0) -- (3,0);
\node [above] at (2,0) {$d$};
\draw (-4,0.2) -- (-4.0, -0.2);
\node [below] at (-4,0) {$\beta^\alpha_{j-1}(0)$};
\draw (0,0.2) -- (0, -0.2);
\node [below] at (0,0) {$\beta^\alpha_j(0)$};
\draw (3,0.2) -- (3, -0.2);
\node [below] at (3,0) {$\beta^\alpha_{j+1}(0)$};
\draw (0,0) circle [radius=1.5];
\node [right] at (1.1,1.1) {$\Gamma$};
\end{tikzpicture} 
\caption{$\Gamma$}
 \label{spectrum}
\end{figure}

Suppose the lowest quasi-periodic resonance eigenvalue for the domain $D$ lies inside $-1/2<\mu^-(\alpha)<0$. It is noted that in the sequel a large and generic class of domains are identified for which $-1/2<\mu^-(\alpha)$.  The corresponding upper bound on the set $z\in S$ for which $A^\alpha(z)$ is not invertible  is given by 
\begin{eqnarray}
z^\ast=\frac{\mu^-(\alpha)+1/2}{\mu^-(\alpha)-1/2}<0,
\label{upperonS}
\end{eqnarray}
see \eqref{bdsonS}.
Now set
\begin{equation}
r^*=\frac{|\alpha|^2d|z^\ast|}{\frac{1}{1/2-\mu^-}+|\alpha|^2d}.
\label{radiusalphanotzero}
\end{equation}
\begin{theorem}{\rm Separation of spectra and radius of convergence for $\alpha\in Y^\star$, $\alpha\not=0$.}\\
\label{separationandraduus-alphanotzero}
The following properties  hold for inclusions with domains $D$ that satisfy \eqref{upperonS}:
\begin{enumerate}
\item If $|z|<r^*$ then $\Gamma$ lies in the resolvent of both $A^\alpha(0)$ and $A^\alpha(z)$ and thus separates the spectrum of $A^\alpha(z)$ into two parts given by the component of spectrum of $A^\alpha(z)$ inside $\Gamma$ denoted by $\Sigma'(z)$ and the component exterior to $\Gamma$  denoted by $\Sigma''(z)$. The invariant subspace of $A^\alpha(z)$ associated with $\Sigma'(z)$ is denoted by $M'(z)$ with $M'(z)=P(z)L^2_{\#}(\alpha,Y)$.

\item The projection $P(z)$ is holomorphic for $|z|<r^*$ and $P(z)$ is given by
\begin{eqnarray}
P(z)=\frac{-1}{2\pi i}\oint_\Gamma R(\zeta,z)\,d\zeta.
\label{formula}
\end{eqnarray}
\item The spaces $M'(z)$ and $M'(0)$ are isomorphic for $|z|<r^*$.
\item The power series \eqref{foureleven} converges uniformly for $z\in\mathbb{C}$ inside $|z|<r^*$.

\end{enumerate}
\end{theorem}

Suppose now $\alpha=0$. Recall from Theorem \ref{equiv2} that the limit spectrum for $A^0(0)$ is
$\sigma(A^0(0)) =  \{{\delta'_j}^{-1}\}_{j\in\mathbb{N}}\cup\{{\nu_j}^{-1}\}_{j\in\mathbb{N}}$. For this case take $\Gamma$ to be the closed contour in $\mathbb{C}$ containing an eigenvalue $\beta_j^0(0)$ in $\sigma(A^0(0))$ but no other element of
$\sigma(A^0(0))$ and define 
\begin{eqnarray}
\label{dforperiodic}
d=\inf_{\zeta\in\Gamma}\{\rm{dist}(\zeta,\sigma(A^0(0)))\}.
\end{eqnarray}
Suppose the lowest quasi-periodic resonance eigenvalue for the domain $D$ lies inside $-1/2<\mu^-(0)<0$ and the corresponding upper bound on $S$ is given by 
\begin{eqnarray}
z^\ast=\frac{\mu^-(0)+1/2}{\mu^-(0)-1/2}<0.
\label{upperonSzero}
\end{eqnarray}
Set
\begin{equation}
r^*=\frac{4\pi^2d|z^\ast|}{\frac{1}{1/2-\mu^-}+4\pi^2d}.
\label{radiusalphazero}
\end{equation}
\begin{theorem}{\rm Separation of spectra and radius of convergence for $\alpha=0$.}
\label{separationandraduus-alphazero}
\\
The following properties  hold for inclusions with domains $D$ that satisfy \eqref{upperonSzero}: 
\begin{enumerate}
\item If $|z|<r^*$ then $\Gamma$ lies in the resolvent of both $A^0(0)$ and $A^0(z)$ and thus separates the spectrum of $A^0(z)$ into two parts given by the component of spectrum of $A^0(z)$ inside $\Gamma$ denoted by $\Sigma'(z)$ and the component exterior to $\Gamma$  denoted by $\Sigma''(z)$. The invariant subspace of $A^0(z)$ associated with $\Sigma'(z)$ is denoted by $M'(z)$ with $M'(z)=P(z)L^2_{\#}(\alpha,Y)$.

\item The projection $P(z)$ is holomorphic for $|z|<r^*$ and $P(z)$ is given by
\begin{eqnarray}
P(z)=\frac{-1}{2\pi i}\oint_\Gamma R(\zeta,z)\,d\zeta.
\label{formula}
\end{eqnarray}
\item The spaces $M'(z)$ and $M'(0)$ are isomorphic for $|z|<r^*$.
\item The power series \eqref{foureleven} converges uniformly for $z\in\mathbb{C}$ inside $|z|<r^*$.
\end{enumerate}
\end{theorem}

Next we provide an explicit representation of the integral operators appearing in the series expansion for the eigenvalue group.
\begin{theorem}{\rm Representation of integral operators in the series expansion for eigenvalues}\\
Let $P_3$ be the projection onto the orthogonal complement of $W_1\oplus W_2\oplus\rm{span}\{\Pi\}$ and let $I$ denote the identity on $L^2(\partial D)$,  then the explicit representation for
for the operators $A_n^\alpha$ in the expansion  \eqref{foureleven}, \eqref{fourtwelve} is given by
\begin{eqnarray}
&&A_1^\alpha=[S_D((\tilde{K}_D^{-\alpha})^\ast+\frac{1}{2}I)^{-1}S_D^{-1}P_3+P_1](-\Delta_\alpha)^{-1}\hbox{ \rm and}\nonumber\\
&&A_n^\alpha=S_D((\tilde{K}_D^{-\alpha})^\ast+\frac{1}{2}I)^{-1}S_D^{-1}[S_D((\tilde{K}_D^{-\alpha})^\ast-\frac{1}{2}I)((\tilde{K}_D^{-\alpha})^\ast+\frac{1}{2}I)^{-1}S_D^{-1}]^{n-1}P_3(-\Delta_\alpha)^{-1}.\nonumber\\
\label{terms}
\end{eqnarray}
\label{reptheorem1}
\end{theorem}

We have a corollary to Theorems \ref{separationandraduus-alphanotzero} and \ref{separationandraduus-alphazero} regarding the error incurred when only finitely many terms of the series \ref{foureleven} are calculated.
\begin{theorem}{\rm Error estimates for the eigenvalue expansion}.\\
\begin{enumerate}
\label{errorestm}
\item Let $\alpha \neq 0$, and suppose $D$, $z^*$, and $r^*$ are as in Theorem \ref{separationandraduus-alphanotzero}. Then the following error estimate for the series \eqref{foureleven} holds for $|z|<r^*$:
\begin{equation}
\left |\hat{\beta}^{\alpha}(z) - \sum \limits_{n = 0}^{p} z^n \beta^{\alpha}_n \right | \leq \frac{d|z|^{p+1}}{(r^*)^p(r^* - |z|)}.
\end{equation}
\item Let $\alpha = 0$, and suppose $D$, $z^*$, and $r^*$ are as in Theorem \ref{separationandraduus-alphazero}. Then the following error estimate for the series \eqref{foureleven} holds for $|z|<r^*$:
\begin{equation}
\left |\hat{\beta}^{0}(z) - \sum \limits_{n = 0}^{p} z^n \beta^{0}_n \right | \leq \frac{d|z|^{p+1}}{(r^*)^p(r^* - |z|)}.
\end{equation}
\end{enumerate}
\end{theorem}

We summarize results in the following theorem.
\begin{theorem}
\label{maintheorem}
The Bloch eigenvalue problem \eqref{Eigen1} is defined for the coupling constant $k$ extended into
the complex plane and the operator $-\nabla\cdot(k\chi_{Y\setminus D}+\chi_D)\nabla$  with domain $H_{\#}^1(\alpha,Y)$ is holomorphic for $k\in\mathbb{C}\setminus Z$. The associated Bloch spectra is given by the eigenvalues $\lambda_j(k,\alpha)=(\beta_j^\alpha(1/k))^{-1}$, for $j\in\mathbb{N}$. For $\alpha\in Y^\star$ fixed, the eigenvalues are of finite multiplicity. Moreover for each $j$ and $\alpha\in Y^\star$, the eigenvalue group is analytic within a neighborhood of infinity
containing the disk $|k|>{r^*}^{-1}$ where $r^*$ is given by \eqref{radiusalphanotzero} for $\alpha\not=0$ and by \eqref{radiusalphazero} for $\alpha=0$.
\end{theorem}

The proofs of Theorems \ref{separationandraduus-alphanotzero}, \ref{separationandraduus-alphazero}  and \ref{errorestm} are given in section \ref{derivation}. The proof of Theorem \ref{reptheorem1} is given in section \ref{leading-order}.

\section{Radius of Convergence and Separation of Spectra for Periodic Scatterers of General Shape}
\label{radiusgeneralshape}

We start by identifying an explicit condition on the inclusion geometry that guarantees a lower bound $\mu^-$ on the quasi-periodic spectra that holds uniformly for $\alpha \in Y^*$, i.e., $-\frac{1}{2} < \mu^- \leq \mu^-(\alpha)=\min_i\{\mu_i\} \leq \frac{1}{2}$ .  

Let $D \Subset Y$ be a union of simply connected sets (inclusions) $D_i$, $i=1,\ldots,N$  with $C^2$ boundary.  Recall that, for any eigenpair $(\mu, w)$ of $T|_{W_3}$ and all $v\in H^1_{\#}(\alpha, Y)$,

\begin{equation}
\frac{1}{2} \int \limits_{Y\setminus D} \nabla w \cdot \nabla \bar{v} - \frac{1}{2} \int \limits_{D} \nabla w \cdot \nabla \bar{v} = \mu \int \limits_{Y} \nabla w \cdot \nabla \bar{v}.
\end{equation}

Adding $\frac{1}{2}\int \limits_{Y} \nabla w \cdot \nabla \bar{v}$ to both sides yields

\begin{equation}
\int \limits_{Y\setminus D} \nabla w \cdot \nabla \bar{v} = (\mu + \frac{1}{2}) \int \limits_{Y} \nabla w \cdot \nabla \bar{v}.
\end{equation}

We will show that there exists a $\rho > 0$ such that $\mu_i + \frac{1}{2} > \rho$ independent of $i \in \mathbb{N}$ and $\alpha \in Y^*$.  If such a $\rho$ exists, then clearly $\mu_i > \rho - \frac{1}{2}$ for all $i$ and $\alpha$, providing an explicit lower bound $\mu^- = \rho-\frac{1}{2}$ satisfying the desired inequality.

\begin{theorem}
\label{lowerboundrho}
Let $\mu^-(\alpha)$ be the lowest eigenvalue of $T$ in $W_3 \subset H^1_{\#}(\alpha, Y)$.  Suppose there is a $\theta >0$ such that for all $u \in W_3$ we have

\begin{equation}
\label{thetaineq}
\| \nabla u\|_{L^2(Y\setminus D)}^2 \geq \theta \| \nabla u \|_{L^2(D)}^2.
\end{equation}

Let $\rho = \min \{ \frac{1}{2}, \frac{\theta}{2} \}$.  Then $\mu^-(\alpha) + \frac{1}{2} > \rho$ for all $\alpha \in Y^*$.
\end{theorem}
\begin{proof}
We proceed by contradiction: suppose that $\mu^-(\alpha) + \frac{1}{2} < \frac{1}{2}$ and $\mu^-(\alpha) + \frac{1}{2} < \frac{\theta}{2}$.  Let $u^-$ be the normalized eigenvector of $T$ associated with $\mu^-(\alpha)$.  Then we have
\begin{equation}
\label{sevenfour}
\int \limits_{Y\setminus D} | \nabla u^-|^2 < \frac{1}{2}
\end{equation}
and
\begin{equation}
\frac{\theta}{2} > \int \limits_{Y \setminus D} | \nabla u^-|^2 \geq \theta \int \limits_{D} | \nabla u^-|^2.
\end{equation}
Thus we have
\begin{equation}
\label{sevensix}
\int \limits_{D} | \nabla u^-|^2 < \frac{1}{2}.
\end{equation}
Inequalities \eqref{sevenfour} and \eqref{sevensix} yield
$$\| \nabla u^- \|_{L^2(Y)}^2 <1.$$
But $u^-$ was normalized so that
$$\| \nabla u^- \|_{L^2(Y)}^2 = 1 \text{,}$$
completing the proof.
\end{proof}

Clearly the parameter $\theta$ is a geometric descriptor for $D$. The class of periodic distributions of inclusions for which Theorem \eqref{lowerboundrho} holds for a fixed positive value of $\theta$ is denoted by $P_\theta$ and we have the corollary given by:

\begin{corollary}
\label{theta}
For every inclusion domain $D$ belonging to $P_\theta$ Theorems \ref{separationandraduus-alphazero} through \ref{maintheorem} hold with $z^*$ replaced with $z_\theta^+$ given by
\begin{eqnarray}
z_\theta^+=\frac{\mu^-+1/2}{\mu^--1/2}<0,
\label{thetaclass}
\end{eqnarray}
where $\mu^-=\min\{\frac{1}{2},\frac{\theta}{2}\}-\frac{1}{2}$.
\end{corollary}

Now we introduce a wide class of inclusion shapes with $\theta>0$ that satisfy \eqref{thetaineq}.  
Consider an inclusion domain $D=\cup_{i=1}^N D_i$. Suppose we can surround each $D_i$ by a buffer layer $R_i$ so that each inclusion $D_i$ together with its buffer does not intersect with the any of the other buffered inclusions, i.e., $D_i\cup R_i\cap D_j\cup R_j=\emptyset$, $i\not=j$. The set of such inclusion domains will be called {\em buffered geometries}, see Figure \ref{plane2}.  We now denote the operator norm for the Dirichlet to Neumann map for each inclusion by $\Vert DN_i\Vert$ and the Poincare constant for each buffer layer by $C_{R_i}$ and we have the following theorem.

\begin{theorem}
The buffered geometry lies in $P_\theta$ provided
\begin{eqnarray}
\theta^{-1}=\max_i\{(1+C_{R_i})\Vert DN_i\Vert\}<\infty.
\label{setcriteria}
\end{eqnarray}
\label{pthetabuffered}
\end{theorem}

\begin{proof}
To prove this theorem it suffices to consider one of the components $D_i$ denoted by $D$ and its buffer $R_i$ denoted by $R$. The union of inclusion and buffer is denoted by $D'=D\cup R$. We now show for any function $w' \in H^1(R)$ there is a $w \in H^1(D')$ such that
$$w(x) = w'(x) \text{,  } x \in R$$
and

\begin{equation}
\label{sevenseven}
\int \limits_D |\nabla w|^2 dx \leq \theta^{-1} \int \limits_R |\nabla w'|^2,
\end{equation}
where $\theta^{-1}=\{1+C_{R}\Vert DN\Vert\}$ and $DN$ is the Dirichlet to Neumann map for $D$.

Let $w \in H^1(D')$ such that $w = w'$ in $R$ and $\Delta w = 0$ in $D$ with boundary condition $w|_{\partial D} = w'$.  Note that since $w$ is harmonic in $D$, we have
$$\int \limits_{\partial D} \partial_\nu w d\sigma =0 \text{,}$$
where $\nu$ is the outward pointing normal vector on $\partial D$.  Thus

\begin{equation}
\begin{array}{lcl}
\int \limits_D |\nabla w|^2 & = & \int \limits_{\partial D} \partial_\nu w \bar{w} = \int \limits_{\partial D} \partial_\nu w \overline{(w - (w')^*)}\\
\\
& = & \int \limits_{\partial D} \partial_\nu w \overline{(w' - (w')^*)} \text{,}
\end{array}
\end{equation}

where $(w')^*$ is the average of $w'$ over $R$, given by

\begin{equation}
(w')^* = \frac{1}{|R|} \int \limits_R w' dx.
\end{equation}

Taking $DN$ as the Dirichlet-to-Neumann map on $H^{1/2}(\partial D)$, we have

\begin{equation}
\label{seventen}
\begin{array}{lcl}
\int \limits_{\partial D} \partial_\nu w \overline{(w' - (w')^*)} & \leq & {}_{H^{-1/2}(\partial D)} \langle DN|_{\partial D} w, w' - (w')^* \rangle_{H^{1/2}(\partial D)} \\
\\
& = & {}_{H^{-1/2}(\partial D)} \langle DN|_{\partial D} [w' - (w')^*], w' - (w')^* \rangle_{H^{1/2}(\partial D)} \\
\\
& \leq & \| DN|_{\partial D} \| \| w' - (w')^* \|_{H^{1/2}(\partial D)}^2.
\end{array}
\end{equation}
The second line of \eqref{seventen} holds since $w = w'$ on $\partial D$ and $Ker(DN)$ is simply the constant functions on $\partial D$.  Let $C_R$ be the Poincar\'{e} constant of $R$, i.e.

\begin{equation}
\| q - (q)^*\|_{L^2(R)}^2 \leq C_R \| \nabla q\|_{L^2(R)}^2
\end{equation}

for all $q \in H^1(R)$.  Then we calculate

\begin{equation}
\label{seventwelve}
\begin{array}{lcl}
\| w' - (w')^* \|_{H^{1/2}(\partial D)}^2 & \leq & \| w' - (w')^* \|_{H^{1/2}(\partial R)}^2 \\
\\
& = &\inf \limits_{v|_{\partial R} = w' - (w')^*} \|v \|^2_{H^1(R)}\\
\\
& \leq & \| w' - (w')^* \|_{H^1(R)}^2 \leq (1+C_R) \| \nabla w'\|_{L^2(R)}^2.
\end{array}
\end{equation}
Substituting the last line of \eqref{seventwelve} into the last line of \eqref{seventen} and setting $\theta^{-1} = \| DN\|(1+C_R)$, we obtain inequality \eqref{sevenseven} as desired.

Let $u \in W_3$, and set $w' = u$ in $R$.  Then the $w$ arising from the above theorem is a harmonic function in $D$ satisfying $w|_{\partial D} = u$.  Since $u$ is also harmonic in D, we have that $u = w$ in $D$ by uniqueness of solutions to Laplace's equation with Dirichlet boundary conditions, and inequality \eqref{sevenseven} becomes
\begin{equation}
\theta \int \limits_D | \nabla u|^2 \leq \int \limits_R | \nabla u |^2 \leq \int \limits_{Y\setminus D} | \nabla u |^2 \text{.}
\end{equation}

\end{proof}


\section{Radius of Convergence and Separation of Spectra  for Disks}
\label{radiusmultiplescatterers}
We now consider Bloch spectra for crystals in $\mathbb{R}^2$ with each  period cell containing an identical random distribution of $N$ disks $D_i$, $i=1,\ldots, N$ of radius $a$. We suppose that the  smallest distance separating the disks is $t_d>0$. The buffer layers $R_i$ are annuli with inner radii $a$ and outer radii $b=a+t$ where $t\leq t_d/2$ and  is chosen so that the collection of buffered disks lie within the period cell. For this case the constant $\theta$ is computed in \cite{Bruno} and is given by
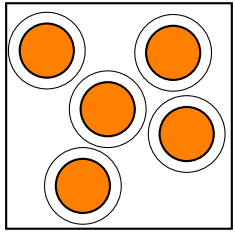
\begin{figure} 
\centering
\begin{tikzpicture}[xscale=0.6,yscale=0.6]
\draw [thick] (-2,-2) rectangle (3,3);
\draw [fill=orange,thick] (0.25,0.65) circle [radius=0.6];
\draw (0.25,0.65) circle [radius=0.85];
\draw [fill=orange,thick] (-1.1,1.95) circle [radius=0.6];
\draw (-1.1,1.95) circle [radius=0.85];
\draw [fill=orange,thick] (-0.3,-1.05) circle [radius=0.6];
\draw (-0.3,-1.05) circle [radius=0.85];
\draw [fill=orange,thick] (2.0,0.10) circle [radius=0.6];
\draw  (2.0,0.10) circle [radius=0.85];
\draw [fill=orange,thick] (1.7,1.9) circle [radius=0.6];
\draw  (1.7,1.9) circle [radius=0.85];

\end{tikzpicture} 
\caption{\bf Random buffered suspension.}
 \label{plane2}
\end{figure}

\begin{equation}
\label{thetaofb}
\theta = \frac{b^2 - a^2}{b^2 + a^2}.
\end{equation}
Since $a<b$, we have that
\begin{equation}
\label{thetabounds}
0 < \theta <1.
\end{equation}
We also note that when $D_i$ is a disc of radius $a>0$, we can recover an explicit formula for $d$ from equation \ref{dist}.  In particular, any eigenvalue $\beta_j^{\alpha}(0)$ of $-\Delta_D^{-1}$, for $\alpha\not=0$. may be written
\begin{equation}
\label{dirichletvaluedisc}
\beta_j^{\alpha}(0) = \left (\frac{\eta_{n,k}}{a} \right )^{-2} \text{,}
\end{equation}
where $\eta_{n,k}$ is the $k$th zero of the $n$th Bessel function $J_n(r)$.  Let $\tilde{\eta}$ be the minimizer of
\begin{equation}
\min \limits_{m,j \in \mathbb{N}} |(\eta_{n,k})^{-2} - (\eta_{m,j})^{-2}|.
\end{equation} 
Then we may choose $\Gamma$ from section \ref{radius} so that
\begin{equation}
d = \frac{1}{2}|(\frac{a}{\eta_{n,k}})^2 - (\frac{a}{\tilde{\eta}})^2|.
\end{equation}
We apply explicit form for $\theta$ to obtain a formula for $r^*$ in terms of $a$,$b$, $d$ given above, and $\alpha$.  Recall that $\rho$ from Theorem \ref{lowerboundrho} is given by $\rho = \min \{\frac{1}{2}, \frac{\theta}{2} \}$.  In light of inequality \eqref{thetabounds}, we have that
\begin{equation}
\label{rhoforab}
\rho = \frac{1}{2} \left ( \frac{b^2-a^2}{b^2+a^2} \right ),
\end{equation}
and we calculate the lower bound $\mu^-$:
\begin{equation}
\label{lowermuforab}
\mu^- = \rho - \frac{1}{2} = -\frac{a^2}{b^2+a^2}.
\end{equation}
Recalling that
$$|z^*|\leq|z^+| = \frac{\mu^- + 1/2}{1/2-\mu^-} \text{ , }$$
we obtain an explicit radius of convergence $r^*$ in terms of $a$, $b$, $\eta_{n,k}$, $\tilde{\eta}$, and $\alpha$ for $\alpha \neq 0$,
\begin{equation}
\label{ralphaforab}
r^* = \frac{|\alpha|^2|(\frac{a}{\eta_{n,k}})^2 - (\frac{a}{\tilde{\eta}})^2|(b^2-a^2)}{4(b^2+a^2) + |\alpha|^2|(\frac{a}{\eta_{n,k}})^2 - (\frac{a}{\tilde{\eta}})^2|(b^2+3a^2)}.
\end{equation}

When $\alpha=0$  Theorem \ref{equiv2} shows that the limit spectrum consists of a component given by the roots $\nu_{0k}$ of
\begin{eqnarray}
1=\nu\sum_{k\in\mathbb{N}}\frac{a_{0k}^2}{\nu-(\eta_{0k}/a)^2},
\label{extraspec}
\end{eqnarray}
where $a_{0k}=\int_D u_{0k}\,dx$ are averages of the rotationally symmetric normalized eigenfunctions $u_{0k}$ given by
\begin{eqnarray}
u_{0k}=J_0(r\eta_{0k}/a)/(a\sqrt{\pi}J_1(\eta_{0k})).
\label{rotintavg}
\end{eqnarray}
The other component is comprised of the eigenvalues exclusively associated with mean zero eigenfunctions. The collection of these eigenvalues is given by $\{\cup_{n\not=0,k}(\eta_{nk}/a)^2\}$
The elements $\lambda_{nk}$ of the spectrum $\sigma(A^0(0))$ are given by the set $\{\cup_{n\not=0,k}(\eta_{nk}/a)^2\}\cup\{\cup_k\nu_{0k}\}$. Now fix an element $\lambda_{nk}$ and  let $\tilde{\eta}$ be the minimizer of
\begin{equation}
\min \limits_{m,j \in \mathbb{N}} |(\lambda_{n,k})^{-1} - (\lambda_{m,j})^{-1}|.
\end{equation} 
Then as before we may choose $\Gamma$ from section \ref{radius} so that
\begin{equation}
d = \frac{1}{2}|({\lambda_{n,k}}^{-1} - {\tilde{\eta}}^{-1}|
\end{equation}
and in terms of $a$, $b$, $\lambda_{n,k}$, and $\tilde{\eta}$ for $\alpha = 0$:
\begin{equation}
\label{ralphaforabzero}
r^* = \frac{4\pi^2|(\lambda_{n,k})^{-1} - {\tilde{\eta}}^{-1}|(b^2-a^2)}{4(b^2+a^2) + 4\pi^2|({\lambda_{n,k}})^{-1} - {\tilde{\eta}}^{-1}|(b^2+3a^2)}.
\end{equation}

The collection of suspensions of $N$ buffered disks is an example of a class of buffered inclusion geometries and collecting results we have the following:

\begin{corollary}
\label{theta2}
For every suspension of buffered disks with $\theta$ given by  \eqref{thetaofb}: Theorem \ref{separationandraduus-alphanotzero} holds with $r^*$ given by \eqref{ralphaforab} for $\alpha\in Y^\star$, $\alpha\not=0$, and Theorem \ref{separationandraduus-alphazero} holds with $r^*$ given by \eqref{ralphaforabzero} for $\alpha=0$.
\end{corollary}

\section{Layer Potential Representation of Operators in Power Series}
\label{leading-order}

In this section we identify explicit formulas for the operators $A^\alpha_n$ appearing in the power series  \eqref{fourtwelve}. It is shown that $A^{\alpha}_n$, $n \neq 0$ can be expressed in terms of  integral operators associated with layer potentials and we establish Theorem \ref{reptheorem1}.  

Recall that $A^{\alpha}(z) - A^{\alpha}(0)$ is given by
 
 \begin{equation}
 (z P_1 + \sum_{-\frac{1}{2} < \mu_i < \frac{1}{2}}z[(1/2 + \mu_i) + z(1/2-\mu_i)) ^{-1}P_{\mu_i})(-\Delta_{\alpha}^{-1}).
 \end{equation}
 
 Factoring $(1/2+\mu_i)^{-1}$ from the second summand, we expand in power series
 
 \begin{equation}
 [(1/2 + \mu_i) + z(1/2-\mu_i)]^{-1} = (1/2 + \mu_i)^{-1}\sum \limits_{n=0}^{\infty}z^n\left (\frac{\mu_i-1/2}{\mu_i + 1/2} \right )^n \text{,}
 \end{equation}
 
and
 
  \begin{equation}
A^{\alpha}(z) - A^{\alpha}(0) = (z P_1 + \sum \limits_{n=1}^{\infty}z^{n}\sum_{-\frac{1}{2} < \mu_i < \frac{1}{2}}(\mu_i +1/2)^{-1} \left (\frac{\mu_i-1/2}{\mu_i + 1/2} \right)^{n-1}  P_{\mu_i} P_3)(-\Delta_{\alpha}^{-1}).
 \end{equation}
 
 It follows that
 \begin{equation}
 \label{AnProjections}
 A_1^{\alpha}  =  [P_1 + \sum \limits_{-\frac{1}{2} < \mu_i < \frac{1}{2}} (1/2 + \mu_i)^{-1} P_{\mu_i}P_3](-\Delta_{\alpha}^{-1})
 \end{equation}
{and}
\begin{equation}
 A_n^{\alpha}  =  [\sum \limits_{-\frac{1}{2} < \mu_i < \frac{1}{2}}(\mu_i +1/2)^{-1} \left (\frac{\mu_i-1/2}{\mu_i + 1/2} \right)^{n-1}  P_{\mu_i} P_3](-\Delta_{\alpha}^{-1}).
 \label{AnProjections2}
 \end{equation}
 
 Recall also that we have the resolution of the identity
 
 \begin{eqnarray}
 I_{H_{\#}^1(\alpha,Y)} = P_1+P_2+P_3\hbox{ with } P_3=\sum \limits_{-\frac{1}{2} < \mu_i < \frac{1}{2}} P_{\mu_i} \text{,}
 \label{residentity}
 \end{eqnarray}
 
 and the spectral representation
 
 \begin{equation}
 \begin{array}{lcl}
 \langle Tu,v \rangle & = & \langle (S_D (\tilde{K}_D^{-\alpha})^*S_D^{-1}) P_3u + \frac{1}{2} P_1u - \frac{1}{2} P_2u, v \rangle \\
 \\ 
 & = & \langle \sum \limits_{-\frac{1}{2} < \mu_i < \frac{1}{2}} \mu_i P_{\mu_i}u + \frac{1}{2} P_1u - \frac{1}{2} P_2u, v \rangle.
 \end{array}
 \end{equation}
 
 Adding $\frac{1}{2} I$ to both sides of the above equation, we obtain
 
 \begin{equation}
 \begin{array}{lcl}
 \langle (T+ \frac{1}{2}I)u,v \rangle & = & \langle (\sum \limits_{-\frac{1}{2} < \mu_i < \frac{1}{2}} (\mu_i + \frac{1}{2})P_{\mu_i} +P_1)u,v\rangle \\
  \\ 
  & = &  \langle ((S_D(\tilde{K}_D^{-\alpha})^*S_D^{-1} +\frac{1}{2}P_3)P_3 +P_1)u,v\rangle \\
  \\
  & = & \langle ((S_D((\tilde{K}_D^{-\alpha})^*+ \frac{1}{2} \tilde{I})S_D^{-1})P_3 +P_1)u,v\rangle \text{,}
  \end{array}
  \label{identproj}
  \end{equation}
 
 where $\tilde{I}$ is the identity on $H^{-1/2}(\partial D)$.  Now from \eqref{identproj} we see that

 \begin{equation}
 \begin{array}{lcl}
 \label{A1spectrumtolayers}
 \sum \limits_{-\frac{1}{2} < \mu_i < \frac{1}{2}} (\frac{1}{2} + \mu_i)^{-1} P_{\mu_i} P_3 & = & (S_D(\tilde{K}_D^{-\alpha})^*S_D^{-1} +\frac{1}{2}P_3)^{-1}P_3 \\
 \\
 & = & (S_D ((\tilde{K}_D^{-\alpha})^* + \frac{1}{2}\tilde{I})S_D^{-1})^{-1}P_3\\
 \\
 & = & (S_D ((\tilde{K}_D^{-\alpha})^* + \frac{1}{2}\tilde{I})^{-1}S_D^{-1})P_3.
 \end{array}
 \end{equation}
 
 Combining the first line of \eqref{AnProjections} and \eqref{A1spectrumtolayers}, we obtain
 
 \begin{equation}
 \label{A1Layers}
 A_1^\alpha=[S_D((\tilde{K}_D^{-\alpha})^*+\frac{1}{2}\tilde{I})^{-1}S_D^{-1}P_3+P_1](-\Delta_\alpha)^{-1}.
 \end{equation}
 
 We now turn to the higher-order terms.  By the mutual orthogonality of the projections $P_{\mu_i}$, we have that
 
\begin{equation}
\begin{array}{lcl}
\label{separateresonances}
 \sum \limits_{-\frac{1}{2} < \mu_i < \frac{1}{2}} (\mu_i +1/2)^{-1} \left (\frac{\mu_i-1/2}{\mu_i + 1/2} \right)^{n-1}  P_{\mu_i} \\
 \\
 = \left( \sum \limits_{-\frac{1}{2} < \mu_i < \frac{1}{2}} (1/2 + \mu_i)^{-1}P_{\mu_i} \right ) \left( \sum \limits_{-\frac{1}{2} < \mu_i < \frac{1}{2}} \left( \frac{\mu_i-1/2}{\mu_i +1/2} \right) P_{\mu_i} \right)^{n-1}\\
  \\
 =  \left( \sum \limits_{-\frac{1}{2} < \mu_i < \frac{1}{2}} (1/2 + \mu_i)^{-1}P_{\mu_i} \right ) \left( \sum \limits_{-\frac{1}{2} < \mu_i < \frac{1}{2}} (\mu_i-1/2) P_{\mu_i} \right)^{n-1} \left( \sum \limits_{-\frac{1}{2} < \mu_i < \frac{1}{2}} (\mu_i +1/2) P_{\mu_i} \right)^{1-n}.
 \end{array}
 \end{equation}
 
 As above, we have that
 \begin{equation}
 \begin{array}{lcl}
 \label{resonancetolayersAn}
 \sum \limits_{-\frac{1}{2} < \mu_i < \frac{1}{2}} (1/2 + \mu_i)^{-1}P_{\mu_i} P_3 & = & S_D((\tilde{K}_D^{-\alpha})^*+\frac{1}{2}\tilde{I})^{-1}S_D^{-1}P_3,\\
 \\
  \sum \limits_{-\frac{1}{2} < \mu_i < \frac{1}{2}} (1/2 + \mu_i)P_{\mu_i} P_3 & = & S_D((\tilde{K}_D^{-\alpha})^*+\frac{1}{2}\tilde{I})S_D^{-1}P_3,\\
  \\
   \sum \limits_{-\frac{1}{2} < \mu_i < \frac{1}{2}} (\mu_i - 1/2)P_{\mu_i} P_3 & = & S_D((\tilde{K}_D^{-\alpha})^*- \frac{1}{2}\tilde{I})S_D^{-1}P_3.
   \end{array}
   \end{equation}
 
 Combining \eqref{resonancetolayersAn}, \eqref{separateresonances}, and \eqref{AnProjections}, we obtain the layer-potential representation for $A_n^{\alpha}$, proving Theorem \ref{reptheorem1}:
 
 \begin{equation}
 A_n^\alpha=S_D((\tilde{K}_D^{-\alpha})^\ast+\frac{1}{2}I)^{-1}S_D^{-1}[S_D((\tilde{K}_D^{-\alpha})^\ast-\frac{1}{2}I)((\tilde{K}_D^{-\alpha})^\ast+\frac{1}{2}I)^{-1}S_D^{-1}]^{n-1}P_3(-\Delta_\alpha)^{-1}.
 \end{equation}
 
 \section{Explicit First Order Correction to the Bloch Band Structure in the High Contrast Limit}
 \label{explicitfirstorder}
In this section we develop explicit formulas for the second term in the power series 
 \begin{equation}
\label{nineone}
\beta^{\alpha}_j (z) = \beta^{\alpha}_j (0) + z\beta^{\alpha}_{j,1} + z^2\beta^{\alpha}_{j.2} + ...
\end{equation}
for simple eigenvalues.
We use the analytic representation of $A^{\alpha}(z)$ and the Cauchy Integral Formula to represent  $\beta^{\alpha}_{j,1}$
\begin{equation}
\begin{array}{lcl}
\label{ninetwo}
\beta^{\alpha}_{j,1} & = & \frac{1}{2\pi im}\text{tr} \oint_\Gamma A^{\alpha}_1 R(0,\zeta) d\zeta \\
\\ 
& = &  \frac{1}{2\pi im}\text{tr}  \left (A^{\alpha}_1 \oint_\Gamma R(0,\zeta) d\zeta \right) \\
\\
& = & \text{tr} \left (A^{\alpha}_1 P(0)\right ) = \frac{1}{m} \sum \limits_{k=1}^m \langle \varphi_k, A^{\alpha}_1 P(0)\varphi_k \rangle_{L^2_{\#} (\alpha,Y)}.
\end{array}
\end{equation}
 Here $P(0)$ is the $L^2_{\#} (\alpha,Y)$ projection onto the eigenspace corresponding to the Dirichlet eigenvalue $(\beta^{\alpha}_j (0))^{-1}$ of $-\Delta$ on $D$.  For simple eigenvalues consider the normalized eigenvector $P(0) \varphi = \varphi$
 and 
 \begin{equation}
\begin{array}{lcl}
\label{ninetwotwo}
\beta^{\alpha}_{j,1} & = \langle \varphi, A^{\alpha}_1 P(0)\varphi \rangle_{L^2_{\#} (\alpha,Y)}.
\end{array}
\end{equation}

 We  apply the integral operator representation of $A_1^{\alpha}$ to deliver an explicit formula for the first order term $\beta_{j,1}^{\alpha}$ in the series for $\beta_j^{\alpha}(z)$.  The explicit formula is given by the following theorem.
 
 \begin{theorem}
Let $\beta^{\alpha}_j(z)$ be an eigenvalue of $A^{\alpha}(z)$.  Then for $|z|<r^*$ there is a $\beta_j(0) \in \sigma(-\Delta^{-1}|_D)$ with corresponding eigenfunction $\varphi_j$ such that
\begin{equation}
\beta^{\alpha}_j (z) = \beta_j (0) +z(\beta_j(0))^2\int_{Y \setminus D} |\nabla v|^2 + z^2\beta^{\alpha}_{j.2} + ...
\end{equation}
Where $v$ takes $\alpha$-quasi periodic boundary conditions on $\partial Y$, is harmonic in $Y \setminus D$, and takes the Neumann boundary condition on $\partial D$ given by
$$\partial_n v |_{\partial D^+} = \partial_n \varphi|_{\partial D^-} \text{,}$$
where $\partial_n$ is the normal derivitave on $\partial D$ with normal vector $n$ pointing into $Y\setminus D$.
\end{theorem}

\begin{remark}
Recall from Theorem \ref{maintheorem} that the eigenvalues $\lambda_j^\alpha(k)=(\beta_j^\alpha(1/k))^{-1}$, for $j\in\mathbb{N}$. The high coupling limit expansion for $\lambda_j^\alpha(k)$ is written in terms of the expansion $\beta_j^{\alpha}(z)=\beta_j (0)+z\beta_{j,1}^{\alpha}+\cdots$ as
\begin{eqnarray}
\lambda_j^\alpha(k)&=&(\beta_j (0))^{-1}-\frac{1}{k}(\beta_j (0))^{-2}\beta_{j,1}^{\alpha}+\cdots\nonumber\\
&=&\lambda_j (0)-\frac{1}{k}\int_{Y \setminus D} |\nabla v|^2 +\cdots,
\label{reversionofseries}
\end{eqnarray}
where $\lambda_j(0)=(\beta_j (0))^{-1}$ is the $j^{th}$ Dirichlet eigenvalue for the Laplacian on $D$.
This naturally agrees with the formula for the leading order terms presented in \cite{AmmariKangLee}.
\end{remark}
 
 \begin{proof}
 Recall from the previous section that
 \begin{equation}
 \begin{array}{lcl}
 A^{\alpha}_1 & = & [S_D((\tilde{K}_D^{-\alpha})^\ast+\frac{1}{2}I)^{-1}S_D^{-1}P_3+P_1](-\Delta_\alpha)^{-1}\\
 \\
 & = & K_1^{\alpha} (-\Delta_\alpha)^{-1},
 \end{array}
 \end{equation}
 
 where $K_1^{\alpha} : = S_D((\tilde{K}_D^{-\alpha})^\ast+\frac{1}{2}I)S_D^{-1}P_3+P_1$.  Moreover,
 
 \begin{equation}
(-\Delta_\alpha)^{-1}f = -\int_Y G^{\alpha}(x,y)f(y)dy.
\end{equation}

Since $\varphi$ is a Dirichlet eigenvector of $D$ with  eigenvalue $(\beta_j(0))^{-1}$ and $\varphi=0$ in $Y\setminus D$, we have

\begin{equation}
\varphi = -\beta_j(0)\chi_D (-\Delta \varphi).
\label{compact}
\end{equation}

Now from \eqref{compact}
\begin{eqnarray}
-\Delta_\alpha^{-1}\varphi & = & \beta_j(0)\int_Y G^{\alpha}(x,y)\chi_D (\Delta_y \varphi)dy\nonumber\\
& = & \beta_j(0)\int_D G^{\alpha}(x,y)(\Delta_y \varphi)dy\nonumber\\
& = & \beta_j(0)(\int_D \nabla_y \cdot (G^{\alpha}(x,y)\nabla_y\varphi) \, dy - \int_D \nabla_yG^{\alpha}(x,y) \cdot \nabla_y\varphi \,dy)\nonumber\\
&=&\beta_j(0)(S_D[\partial_n\varphi_{|_{\scriptscriptstyle{\partial D^-}}}](x)-R(x)),
\label{string}
\end{eqnarray}
where the last equality follows from the divergence theorem and definition of the single layer potential $S_D$ and 
\begin{equation}
R(x)=\int_D \nabla_yG^{\alpha}(x,y) \cdot \nabla_y\varphi \, dy.
\label{R}
\end{equation}
Hence
\begin{eqnarray}
\label{nineeight}
A_1^{\alpha} \varphi & = & K_1^\alpha\beta_j(0)(S_D[\partial_n\varphi_{\scriptscriptstyle{\partial D^-}}](x)-R(x)).
\end{eqnarray}
Now we aply the definition of $K_1^\alpha$ and compute $P_1R(x)$ and $P_3R(x)$. Integrating by parts, we find

\begin{equation}
\begin{array}{lcl}
R(x) & = & \int_D \nabla_yG^\alpha(x,y) \cdot \nabla_y \varphi \,dy\\
\\
& = & \int_D \nabla_y \cdot (\nabla_yG^\alpha(x,y)\varphi)dy - \int_D -\Delta_yG^\alpha(x,y)\varphi \,dy\\
\\
& = & \varphi(x).
\end{array}
\end{equation}

Thus $P_1R(x) = P_3R(x) = 0$ since $\varphi\in W_1$.  Combining this result, \eqref{ninetwo}, \eqref{nineeight}, and the definition of $K^\alpha_1$ we obtain

\begin{equation}
\begin{array}{lcl}
\label{nineten}
\beta^\alpha_{j,1}=\text{tr} \left (A^{\alpha}_1 P(0)\right ) & = & \langle \varphi, A^{\alpha}_1 P(0)\varphi \rangle_{L^2_{\#} (\alpha,Y)}\\
\\
& = & \langle \varphi, \beta_j(0)S_D((\tilde{K}_D^{-\alpha})^\ast+\frac{1}{2}I)^{-1} [\partial_n\varphi_{|_{\scriptscriptstyle{\partial D^-}}}] \rangle_{L^2_{\#} (\alpha,Y)}.
\end{array}
\end{equation}

Let $v \in H^1_{\#}(\alpha,Y)$ be defined
\begin{equation}
v := S_D((\tilde{K}_D^{-\alpha})^\ast + \frac{1}{2}I)^{-1} [\partial_n\varphi_{|_{\scriptscriptstyle{\partial D^-}}}].
\label{v}
\end{equation}
Then $v$ is harmonic in $D$ and $Y\setminus D$, and
\begin{equation}
\partial_n v |_{\partial D^+} = \partial_n \varphi|_{\partial D^-} \text{.}
\label{identbdry}
\end{equation}

On applying \eqref{compact}, \eqref{v}, and \eqref{identbdry}  equation \eqref{nineten} becomes

\begin{eqnarray}
\label{ninethirteen}
\beta^\alpha_{j,1}& = & \beta_j(0) \langle \varphi, v \rangle = - (\beta_j(0))^2\int_Dv\Delta\overline{\varphi}\,dy \\
& = & - (\beta_j(0))^2(\int_{\partial D} \partial_n \varphi_{\scriptscriptstyle |_{\partial D^-}}\bar{v} d\sigma-\int_D \nabla \varphi \cdot \nabla \bar{v} )\nonumber\\
&=& - (\beta_j(0))^2(\int_{\partial D} \partial_n v_{\scriptscriptstyle |_{\partial D^+}}\bar{v} d\sigma-\int_D \nabla \varphi \cdot \nabla \bar{v} ).\nonumber
\end{eqnarray}
Last an integration by parts yields
\begin{eqnarray}
\int_D \nabla \varphi \cdot \nabla \bar{v} & = & \int_D \nabla \cdot(\nabla \bar{v} \varphi) - \Delta \bar{v} \varphi \nonumber \\ 
& = & \int_{\partial D} \partial_n \bar{v}{\scriptscriptstyle |_{\partial D^-}} \varphi d\sigma = 0.
\label{intparts}
\end{eqnarray}
Combining this result with the last line of \eqref{ninethirteen} and integrating by parts a final time reveals a representation of the second term in \eqref{nineone}
\begin{equation}
\beta_{j,1}^\alpha = (\beta_j(0))^2\int_{Y \setminus D} |\nabla v|^2 dx,
\end{equation}
and the theorem follows.
\end{proof}

\section{Derivation of the Convergence Radius and Separation of Spectra}
\label{derivation}

Here we prove Theorems \ref{separationandraduus-alphanotzero} and \ref{separationandraduus-alphazero}.  To begin, we suppose $\alpha\not=0$ and  recall that the Neumann series \eqref{foursix} and consequently \eqref{Project1} and \eqref{foureleven} converge provided that
\begin{equation}
\label{tenone}
\| (A^{\alpha}(z) - A^{\alpha}(0))R(\zeta,0) \|_{\mathcal{L}[L^2_{\#}(\alpha,Y);L^2_{\#}(\alpha,Y)]} <1.
\end{equation}
With this in mind we will compute an explicit upper bound $B(\alpha,z)$ and identify a neighborhood of the origin on the complex plane for which
\begin{equation}
\label{tenoneb}
\| (A^{\alpha}(z) - A^{\alpha}(0))R(\zeta,0) \|_{\mathcal{L}[L^2_{\#}(\alpha,Y);L^2_{\#}(\alpha,Y)]} <B(\alpha,z)<1,
\end{equation}
holds for $\zeta\in\Gamma$.
The inequality $B(\alpha,z)<1$ will be used first to derive a lower bound on the radius of convergence of the power series expansion of the eigenvalue group about $z=0$. It will then be used to provide a lower bound on the neighborhood of $z=0$ where properties 1 through 3 of Theorem \ref{separationandraduus-alphanotzero} hold.

We have the basic estimate given by
\begin{eqnarray}
\label{tenonedouble}
&&\| (A^{\alpha}(z) - A^{\alpha}(0))R(\zeta,0) \|_{\mathcal{L}[L^2_{\#}(\alpha,Y);L^2_{\#}(\alpha,Y)]}\leq \\
&&\| (A^{\alpha}(z) - A^{\alpha}(0))\|_{\mathcal{L}[L^2_{\#}(\alpha,Y);L^2_{\#}(\alpha,Y)]}\|R(\zeta,0) \|_{\mathcal{L}[L^2_{\#}(\alpha,Y);L^2_{\#}(\alpha,Y)]}.\nonumber
\end{eqnarray}
Here $\zeta\in\Gamma$ as defined in Theorem \ref{separationandraduus-alphanotzero} and elementary arguments deliver the estimate
\begin{eqnarray}
\label{tenonedoubleRz}
\|R(\zeta,0) \|_{\mathcal{L}[L^2_{\#}(\alpha,Y);L^2_{\#}(\alpha,Y)]}\leq d^{-1},
\end{eqnarray}
where $d$ is given by \eqref{dist}.

Next we estimate $\| (A^{\alpha}(z) - A^{\alpha}(0))\|_{\mathcal{L}[L^2_{\#}(\alpha,Y);L^2_{\#}(\alpha,Y)]} $. Denote the energy seminorm of  $u$ by

\begin{equation}
\| u \|= \| \nabla u \|_{L^2(Y)}.
\end{equation}
To proceed we introduce the Poincare estimate for functions belonging to $H^1_{\#}(\alpha,Y)$ for $\alpha\not=0$:
\begin{lemma}
\begin{equation}
\label{alpha-poincare}
\| u \|_{L^2(Y)} \leq |\alpha|^{-1}\|u\|.
\end{equation}
\label{poincarealpha}
\end{lemma}
\begin{proof}
A straight forward calculation using \eqref{Greensalpha} gives the upper bound
\begin{eqnarray}
(-\Delta_\alpha^{-1}v,v)_{L^2(Y)}\leq|\alpha|^{-2}\|v\|^2_{L^2(Y)}
\label{spectralbound}
\end{eqnarray}
and we have the Cauchy inequality
\begin{eqnarray}
\|v\|^2_{L^2(Y)}=\langle-\Delta_\alpha^{-1}v,v\rangle\leq\|-\Delta_\alpha^{-1}v\|\|v\|.
\label{identl2}
\end{eqnarray}
Applying \eqref{spectralbound} we get
\begin{eqnarray}
\|-\Delta_\alpha^{-1}v\|=(\langle-\Delta_\alpha^{-1}v,-\Delta_\alpha^{-1}v\rangle)^{1/2}=((-\Delta_\alpha^{-1}v,v))^{1/2}\leq |\alpha|^{-1}\|v\|_{L^2(Y)}
\label{energyl2}
\end{eqnarray}
and the Poincare inequality follows from \eqref{identl2} and \eqref{energyl2}.
\end{proof}
For any $v \in L^2_{\#}(\alpha,Y)$, we  apply \eqref{alpha-poincare} to find
\begin{eqnarray}
\label{tenfive}
&&\| (A^{\alpha}(z) - A^{\alpha}(0)) v\|_{L^2(Y)} \nonumber\\
&&\leq  |\alpha |^{-1}\| (A^{\alpha}(z) - A^{\alpha}(0)) v\|\\
&&= |\alpha |^{-1}\| ((T^{\alpha}_k)^{-1} - (T^{\alpha}_0)^{-1})(-\Delta_\alpha)^{-1} v\|\nonumber\\
&&\leq |\alpha |^{-1}\| ((T^{\alpha}_k)^{-1} - P_2)\|_{\mathcal{L}[H^1_{\#}(\alpha,Y);H^1_{\#}(\alpha,Y)]} \|-\Delta_\alpha^{-1} v\|.\nonumber
\end{eqnarray}
Applying  \eqref{energyl2} and \eqref{tenfive} delivers the upper bound:
\begin{equation}
\label{tenten}
\| (A^{\alpha}(z) - A^{\alpha}(0)) \|_{\mathcal{L}[L^2_{\#}(\alpha,Y);L^2_{\#}(\alpha,Y)]}  \leq  |\alpha |^{-2}\| ((T^{\alpha}_k)^{-1} - P_2)\|_{\mathcal{L}[H^1_{\#}(\alpha,Y);H^1_{\#}(\alpha,Y)]}.
\end{equation}

The next step is to obtain an upper bound on $\| ((T^{\alpha}_k)^{-1} - P_2)\|_{\mathcal{L}[H^1_{\#}(\alpha,Y);H^1_{\#}(\alpha,Y)]}$. For all $v \in H^1_{\#}(\alpha, Y)$, we have

\begin{equation}
\label{teneleven}
\frac{\| ((T^{\alpha}_k)^{-1} - P_2)v\|}{\| v \|} \leq |z|\{w_0 + \sum \limits_{i=1}^{\infty}w_i |(1/2 + \mu_i) + z(1/2-\mu_i)|^{-2}\}^{1/2},
\end{equation}
where $w_0=\|P_1 v\|^2/\|v\|^2$,  $w_i=\|P_i v\|^2/\|v\|^2$, and $w_0+\sum_{i=1}^\infty w_i=1$.
So maximizing the right hand side is equivalent to calculating

\begin{equation}
\begin{array}{lcl}
\max \limits_{w_0+\sum w_i = 1} \{w_0 + \sum \limits_{i=1}^{\infty}w_i |(1/2 + \mu_i) + z(1/2-\mu_i)|^{-2}\}^{1/2}\\
\\
 =  \sup \{1, |(1/2 + \mu_i) + z(1/2-\mu_i)|^{-2}\}^{1/2}.
\end{array}
\end{equation}
Thus we maximize the function
\begin{equation}
f(x) = |\frac{1}{2} + x + z(\frac{1}{2} -x)|^{-2}
\end{equation}
over $x \in [\mu^-(\alpha), {\mu}^+(\alpha)]$ for $z$ in a neighborhood about the origin.  Let $Re(z)=u$, $Im(z)=v$ and we write
\begin{equation}
\begin{array}{lcl}
f(x) & = & |\frac{1}{2} + x + (u+iv)(\frac{1}{2} -x)|^{-2}\\
\\ & = & ((\frac{1}{2}+x+u(\frac{1}{2}-x))^2 + v^2(\frac{1}{2}-x)^2)^{-1}\\
\\
& \leq & (\frac{1}{2}+x+u(\frac{1}{2}-x))^{-2} = g(Re(z),x)\text{,}
\end{array}
\end{equation}
to get the bound
\begin{equation}
\label{tenfifteen}
\| ((T^{\alpha}_k)^{-1} - P_2)\|_{L(H^1_{\#}(\alpha,Y))} \leq |z| \sup\{1, \sup \limits_{x \in [\mu^-(\alpha), \mu^+(\alpha)]} g(u,x)\}^{1/2}.
\end{equation}

We now examine the poles of $g(u,x)$ and the sign of its partial derivative $\partial_x g(u,x)$ when $|u|<1$.  If $Re(z)=u$ is fixed, then $g(u,x) = ((\frac{1}{2} + x) + u(\frac{1}{2} - x))^{-2}$ has a pole when $(\frac{1}{2} + x) + u(\frac{1}{2} - x)=0$.
For $u$ fixed this occurs when
\begin{equation}
\hat{x}=\hat{x}(u)= \frac{1}{2} \left(\frac{1+u}{u-1}\right).
\end{equation}
On the other hand, if $x$ is fixed, $g$ has a pole at
\begin{equation}
u= \frac{\frac{1}{2} + x}{x - \frac{1}{2}}.
\end{equation}
The sign of $\partial_x g$ is determined by the formula
\begin{equation}
\label{teneighteen}
\begin{array}{lcl}
\partial_x g(u,x) & = & {N}/{D} \text{,}
\end{array}
\end{equation}
where $N=-2(1-u)^2x-(1-u^2)$ and $D := ((\frac{1}{2} + x) + u(\frac{1}{2} - x))^4 \geq 0$.  Calculation shows that $\partial_x g<0$ for $x>\hat{x}$, i.e. $g$ is decreasing on $(\hat{x},\infty)$.  Similarly, $\partial_x g>0$ for $x<\hat{x}$ and $g$ is increasing on $(-\infty, \hat{x})$.\\

Now we identify all $u=Re(z)$ for which $\hat{x}=\hat{x}(u)$ satisfies 
\begin{equation}
\hat{x} < \mu^-(\alpha) < 0 \text{.}
\end{equation}
Indeed for such $u$, the function $g(u,x)$ will be decreasing on $[\mu^-(\alpha), \mu^+(\alpha)]$, so that $g(u,\mu^-(\alpha)) \geq g(u,x)$ for all $x \in [\mu^-(\alpha), \bar{\mu}]$, yielding an upper bound for \eqref{tenfifteen}.
\begin{lemma}
\label{identifyu}
The set $U$ of $u \in \mathbb{R}$ for which $-\frac{1}{2} < \hat{x}(u) < \mu^-(\alpha) < 0$ is given by
$$U := [z^*, 1]$$
where
$$-1\leq z^*:=\frac{\mu^-(\alpha)+\frac{1}{2}}{\mu^-(\alpha)-\frac{1}{2}}<0.$$
\end{lemma}
\begin{proof}
Note first that $\mu^-(\alpha)=\inf_{i\in\mathbb{N}}\{\mu_i\}\leq 0$ follows from the fact that zero is an accumulation point for the sequence $\{\mu_i\}_{i\in\mathbb{N}}$ so it follows that $-1\leq z^*$.
Noting $\hat{x} = \hat{x}(u) = \frac{1}{2} \frac{u+1}{u-1}$, we invert and write
\begin{equation}
u = \frac{\frac{1}{2} + \hat{x}}{\hat{x} - \frac{1}{2}}.
\end{equation}
We now show that
\begin{equation}
  z^*\leq u\leq 1
\end{equation}
for $ \hat{x} \leq \mu^-(\alpha)$.  Set $h(\hat{x}) = \frac{\frac{1}{2} +\hat{x}}{\hat{x}-\frac{1}{2}}$.  Then
\begin{equation}
h'(\hat{x}) = \frac{-1}{(\hat{x}-\frac{1}{2})^2} \text{,}
\end{equation}
and so $h$ is decreasing on $(-\infty, \frac{1}{2})$.  Since $\mu^-(\alpha)<\frac{1}{2}$, $h$ attains a minimum over $(-\infty, \mu^-(\alpha)]$ at $x=\mu^-(\alpha)$.  Thus $\hat{x}(u) \leq \mu^-(\alpha)$ implies
\begin{equation}
z^*=\frac{\mu^-(\alpha)+\frac{1}{2}}{\mu^-(\alpha) - \frac{1}{2}} \leq u\leq 1
\end{equation}
as desired.
\end{proof}

Combining Lemma \ref{identifyu} with inequality \eqref{tenfifteen}, noting that $-|z|\leq Re(z) \leq |z|$ and on rearranging terms we obtain the following corollary.
\begin{corollary}
\label{boundAz}
For $|z| < |z^*|$:
\begin{equation}
\| (A^{\alpha}(z) - A^{\alpha}(0)) \|_{\mathcal{L}[L^2_{\#}(\alpha,Y);L^2_{\#}(\alpha,Y)]}  \leq |\alpha |^{-2} |z| (-|z|-z^*)^{-1}(\frac{1}{2}-\mu^-(\alpha))^{-1}.
\end{equation}
\end{corollary}
From Corollary \ref{boundAz}, \eqref{tenonedouble}, and  \eqref{tenonedoubleRz}   we easily see that
\begin{eqnarray}
\label{tentwentyseven}
&&\| (A^{\alpha}(z) - A^{\alpha}(0))R(\zeta,0) \|_{\mathcal{L}[L^2_{\#}(\alpha,Y);L^2_{\#}(\alpha,Y)]} \leq\\
&&B(\alpha,z)=|\alpha |^{-2} |z| (-|z|-z^*)^{-1}(\frac{1}{2}-\mu^-(\alpha))^{-1}d^{-1}. \nonumber
\end{eqnarray}
a straight forward calculation shows that $B(\alpha,z)<1$ for
\begin{equation}
|z| < r^*:= \frac{|\alpha|^2d|z^*|}{\frac{1}{\frac{1}{2} - \mu^-(\alpha)} + |\alpha|^2d}
\end{equation}
and property 4 of Theorem \ref{separationandraduus-alphanotzero} is established since $r^* < |z^*|$.

Now we establish properties 1 through 3 of Theorem \ref{separationandraduus-alphanotzero}.
First note that inspection of  \eqref{foursix} shows that if \eqref{tenone} holds and if $\zeta\in\mathbb{C}$ belongs to the resolvent of $A^\alpha(0)$  then it also belongs to the resolvent of $A^\alpha(z)$. Since \eqref{tenone} holds for $\zeta\in\Gamma$ and $|z|<r^*$, property 1 of Theorem \ref{separationandraduus-alphanotzero} follows.
Formula \eqref{Project1} shows that $P(z)$ is analytic in a neighborhood of $z=0$ determined by the condition that \eqref{tenone}  holds for $\zeta\in\Gamma$. The set $|z|<r^*$ lies inside this neighborhood
and property 2 of Theorem \ref{separationandraduus-alphanotzero} is proved. The isomorphism expressed in property 3 of Theorem \ref{separationandraduus-alphanotzero}  follows directly from Lemma 4.10
(\cite{KatoPerturb}, Chapter I \S 4) which is also valid in a Banach space.

The proof of \ref{separationandraduus-alphazero} proceeds along identical lines.
To prove  Theorem \ref{separationandraduus-alphazero}, we need  the  following Poincar\'{e} inequality for $H^1_{\#}(0,Y)$.
\begin{lemma}
\begin{equation}
\label{poincarealphazero}
\|v\|_{L^2_{\#}(0,Y)} \leq \frac{1}{2\pi}\|v\|.
\end{equation}
\label{poincalphazero}
\end{lemma}

This inequality is established using \eqref{Greensalphazero} and proceeding using the same steps as in the proof of Lemma \ref{poincarealpha}.
Using \eqref{poincarealphazero} in place of \eqref{alpha-poincare} we argue as in the proof of Theorem \ref{separationandraduus-alphanotzero}  to show that 
\begin{equation}
\| (A^{0}(z) - A^{0}(0))R(\zeta,0) \|_{\mathcal{L}[(L^2_{\#}(0,Y);L^2_\#(0,Y)]} <1
\end{equation}
holds provided $|z| < r^*$, where $r^*$ is given by \eqref{radiusalphazero}.   This establishes Theorem \ref{separationandraduus-alphazero}.

The error estimates presented in Theorem \ref{errorestm} are easily recovered from the arguments in (\cite{KatoPerturb} Chapter II, \S 3); for completeness, we restate them here.  We begin with the following application of Cauchy inequalities to the coefficients $\beta^{\alpha}_n$ of \eqref{foureleven} from (\cite{KatoPerturb} Chapter II, \S 3, pg 88):
\begin{equation}
\left | \beta^{\alpha}_n \right | \leq d(r^*)^{-n}.
\end{equation}
It follows immediately that, for $|z|<r^*$,
\begin{equation}
\left |\hat{\beta}^{\alpha}(z) - \sum \limits_{n = 0}^{p} z^n \beta^{\alpha}_n \right | \leq \sum \limits_{n=p+1}^{\infty} |z|^n |\beta^{\alpha}_n| \leq \frac{d|z|^{p+1}}{(r^*)^p(r^* - |z|)}\text{,}
\end{equation}
completing the proof.

For completeness we establish the boundedness and compactness of the operator $B^\alpha(k)$.
\begin{theorem}
\label{bounded}
The operator $B^\alpha(k): L^2_{\#}(\alpha,Y) \longrightarrow H^1_{\#}(\alpha,Y)$ is  bounded for $k\not\in Z$.
\end{theorem}
To prove the theorem for $\alpha\not=0$ we observe for $v\in L^2_\#(\alpha,Y)$ that
\begin{eqnarray}
&&\Vert B^\alpha(k) v\Vert=\vert(T_k^\alpha)^{-1}(-\Delta_\alpha)^{-1} v\Vert\leq\nonumber\\
&&\leq \| ((T^{\alpha}_k)^{-1}\|_{\mathcal{L}[H^1_{\#}(\alpha,Y);H^1_{\#}(\alpha,Y)]} \|-\Delta_\alpha^{-1} v\|\nonumber\\
&&\leq |\alpha|^{-1}\| ((T^{\alpha}_k)^{-1}\|_{\mathcal{L}[H^1_{\#}(\alpha,Y);H^1_{\#}(\alpha,Y)]} \Vert v\Vert_{L^2(Y)},
\label{operatorfield}
\end{eqnarray}
where the last inequality follows from \eqref{energyl2}. The upper estimate on $\| ((T^{\alpha}_k)^{-1}\|_{\mathcal{L}[H^1_{\#}(\alpha,Y);H^1_{\#}(\alpha,Y)]} $ is obtained from
\begin{equation}
\label{tenelevenpart2}
\frac{\| T^{\alpha}_k)^{-1}v\|}{\| v \|} \leq \{|z|\hat{w}+\tilde{w}+|\sum \limits_{i=1}^{\infty}w_i |(1/2 + \mu_i) + z(1/2-\mu_i)|^{-2}\}^{1/2},
\end{equation}
where $\hat{w}=\|P_1 v\|^2/\|v\|^2$=, $\tilde{w}=\|P_2v\|^2/\|v\|^2$, $w_i=\|P_i v\|^2/\|v\|^2$. Since $\hat{w}+\overline{w}+\sum_{i=1}^\infty w_i=1$
one recovers the upper bound
\begin{equation}
\label{tenelevenpart2}
\frac{\| T^{\alpha}_k)^{-1}v\|}{\| v \|} \leq M,
\end{equation}
where
\begin{equation}
M= \max\{1, |z|, \sup_{i} \{ |(1/2 + \mu_i) + z(1/2-\mu_i)|^{-1}\}\},
\label{summupperbound}
\end{equation}
and the proof of Theorem \ref{bounded} is complete. An identical proof can be carried out when $\alpha=0$.

\begin{remark}
The Poincare inequalities \eqref{alpha-poincare} and \eqref{poincarealphazero} together with Theorem \ref{bounded} show that
$B^\alpha(k): L^2_{\#}(\alpha,Y) \longrightarrow L^2_{\#}(\alpha,Y)$ is a bounded linear operator mapping $L^2_\#(\alpha,Y)$ into itself. The compact embedding of $H^1_\#(\alpha,Y)$ into $L^2_\#(\alpha,Y)$ shows the operator is compact on $L^2_\#(\alpha,Y)$. 
\label{compact2}
\end{remark}



\section*{Acknowledgements}
This research is supported by AFOSR MURI Grant FA9550-12-1-0489 administered through the University of New Mexico, NSF grant DMS-1211066, and NSF EPSCOR Cooperative Agreement No. EPS-1003897 with additional support from the Louisiana Board of Regents.

\bibliographystyle{plain}
\bibliography{Photonicrefs1}






\end{document}